\documentclass{amsart}
\usepackage{amsfonts}
\usepackage{amssymb}
\usepackage{amsmath}

\newtheorem{theorem}{Theorem}[section]
\theoremstyle{plain}

\newtheorem{remark}{Remark}[section]

\numberwithin{equation}{section}
\input{tcilatex}

\begin{document}
\title[SEQUENCES ASSOCIATED TO ELLIPTIC CURVES]{SEQUENCES ASSOCIATED TO
ELLIPTIC CURVES}
\author{BET\"{U}L GEZER}
\address{Bursa Uludag University, Faculty of Science, Department of
Mathematics, G\"{o}r\"{u}kle, 16059, Bursa-TURKEY}
\email{betulgezer@uludag.edu.tr}
\date{25. 09. 2019.}
\maketitle

\begin{abstract}
Let $E$ be an elliptic curve defined over a field $K$ (with $char(K)\neq 2$)
given by a Weierstrass equation and let $P=(x,y)\in E(K)$ be a point. Then
for each $n$ $\geq 1$ and some $\gamma \in K^{\ast }$ we can write the $x$-
and $y$-coordinates of the point $[n]P$ as%
\begin{equation*}
\lbrack n]P=\left( \frac{\phi _{n}(P)}{\psi _{n}^{2}(P)},\frac{\omega _{n}(P)%
}{\psi _{n}^{3}(P)}\right) =\left( \frac{\gamma ^{2}G_{n}(P)}{F_{n}^{2}(P)},%
\frac{\gamma ^{3}H_{n}(P)}{F_{n}^{3}(P)}\right)
\end{equation*}%
where $\phi _{n},\psi _{n},\omega _{n}\in K[x,y]$, $\gcd (\phi _{n},\psi
_{n}^{2})=1$ and 
\begin{equation*}
F_{n}(P)=\gamma ^{1-n^{2}}\psi _{n}(P),G_{n}(P)=\gamma ^{-2n^{2}}\phi
_{n}(P),H_{n}(P)=\gamma ^{-3n^{2}}\omega _{n}(P)
\end{equation*}%
are suitably normalized division polynomials of $E$. In this work we show
the coefficients of the elliptic curve $E$ can be defined in terms of the
sequences of values $(G_{n}(P))_{n\geq 0}$ and $(H_{n}(P))_{n\geq 0}$ of the
suitably normalized division polynomials of $E$ evaluated at a point $P$ $%
\in E(K)$. Then we give the general terms of the sequences $%
(G_{n}(P))_{n\geq 0}$ and $(H_{n}(P))_{n\geq 0}$ associated to Tate normal
form of an elliptic curve. As an application of this we determine square and
cube terms in these sequences.
\end{abstract}

\section{Introduction}

Let $E$ denote an elliptic curve defined over a field $K$ given by a
Weierstrass equation%
\begin{equation}
E:y^{2}+a_{1}xy+a_{3}y=x^{3}+a_{2}x^{2}+a_{4}x+a_{6}\text{. }  \label{111}
\end{equation}%
For background on elliptic curves, see \cite{JS1} and \cite{JS2}. Let $E(K)$
be the group of $K$-rational points on $E$, let $\mathcal{O}$ denote the
point at infinity, the identity for the group $K$-rational points. Let $K(E)$
denote the function field of $E$ over $K$. Then $z=-x/y\in K(E)$ is a
uniformizer at $\mathcal{O}$ and the invariant differential $\omega =dx%
{\large /}(2y+a_{1}x+a_{3})$ has an expansion as a formal Laurent series in
a formal neighborhood of $\mathcal{O}$ such that%
\begin{equation*}
\omega (z)=(1+a_{1}z+(a_{1}^{2}+a_{2})z^{2}+{\tiny \cdot \cdot \cdot })~dz%
\text{.}
\end{equation*}%
This series has coefficients in $%
\mathbb{Z}
\lbrack a_{1},a_{2},a_{3},a_{4},a_{6}]$, and the uniformizer $z$ and the
differential $\omega $ at $\mathcal{O}$ satisfy $(\omega {\large /}dz)(%
\mathcal{O})=1$. Let $n\geq 1$ be an integer, and let $[n](z)\in $ $K$[[$z$%
]] be the power series defining the multiplication-by-$n$ map on the formal
group of $E$. The $n$-\textit{division polynomial }$F_{n}$ (normalized
relative to the uniformizer $z$)\textit{\ }is the unique function $F_{n}$ $%
\in K(E)$ with divisor $[n]^{-1}(\mathcal{O})-n^{2}(\mathcal{O})$ such that 
\begin{equation*}
\left( \frac{z^{n^{2}}F_{n}}{[n](z)}\right) (\mathcal{O})=1
\end{equation*}%
as defined in \cite[Definition 1]{JS3}, see also \cite{BM1} for details.

If $E$ is an elliptic curve over $%
\mathbb{C}
$, then $E$ has a complex uniformization $\Phi :%
\mathbb{C}
/L\rightarrow E(%
\mathbb{C}
)$, with a lattice $L$ $\subset $ $%
\mathbb{C}
$. The classical $n$-\textit{division polynomial} $\psi _{n}$ of an elliptic
curve $%
\mathbb{C}
/L$ can be expressed in terms of the Weierstrass $\sigma $-function:%
\begin{equation*}
\psi _{n}(z)=\psi _{n}(z,L)=\frac{\sigma (nz,L)}{\sigma (z,L)^{n^{2}}}~\ \ 
\text{for all }n\geq 1\text{,}
\end{equation*}%
where $\sigma (z,L)$ is the Weierstrass $\sigma $-function associated to the
lattice $L$. Moreover, the classical $n$-division polynomial $\psi _{n}$ for
the elliptic curve $E$ evaluated at point $P=(x,y)$ is defined using the
initial values%
\begin{eqnarray*}
\psi _{0}(P) &=&0, \\
\psi _{1}(P) &=&1, \\
\psi _{2}(P) &=&2y+a_{1}x+a_{3}\text{,} \\
\psi _{3}(P) &=&3x^{4}+b_{2}x^{3}+3b_{4}x^{2}+3b_{6}x+b_{8}\text{,} \\
\psi _{4}(P) &=&\psi
_{2}(P)(2x^{6}+b_{2}x^{5}+5b_{4}x^{4}+10b_{6}x^{3}+10b_{8}x^{2} \\
&&+(b_{2}b_{8}-b_{4}b_{6})x+(b_{4}b_{8}-b_{6}^{2}))\text{,}
\end{eqnarray*}%
where the point $P$ correspond to $z\in 
\mathbb{C}
/L$ and $b_{i}$ are the usual quantities \cite[Chapter III.1]{JS1}, and by
the formulas%
\begin{eqnarray*}
\psi _{2n+1}(P) &=&\psi _{n+2}(P)\psi _{n}(P)^{3}-\psi _{n-1}(P)\psi
_{n+1}(P)^{3}\text{, \ \ \ \ \ \ \ \ \ \ \ \ \ \ for }n\geq 2\text{,} \\
\psi _{2n}(P)\psi _{2}(P) &=&\psi _{n-1}(P)^{2}\psi _{n}(P)\psi
_{n+2}(P)-\psi _{n-2}(P)\psi _{n}(P)\psi _{n+1}(P)^{3}\text{, \ \ for }n\geq
3\text{.}
\end{eqnarray*}

Let $P$ $=(x,y)$ be a point of $E(K)$ (with $char(K)\neq 2$), and $n\geq 1$.
The coordinates of the point $[n]P$ can be expressed in terms of the point $%
P $, that is, for some $\gamma \in K$%
\begin{equation}
\lbrack n]P=\left( \frac{\phi _{n}(P)}{\psi _{n}(P)^{2}},\frac{\omega _{n}(P)%
}{\psi _{n}(P)^{3}}\right) =\left( \frac{\gamma ^{2}G_{n}(P)}{F_{n}(P)^{2}},%
\frac{\gamma ^{3}H_{n}(P)}{F_{n}(P)^{3}}\right)  \label{a}
\end{equation}%
where $\phi _{n}$, $\psi _{n}$, $\omega _{n}\in K[x$, $y]$, $\gcd (\phi _{n}$%
, $\psi _{n}^{2})=1$, and 
\begin{equation}
F_{n}(P)=\gamma ^{1-n^{2}}\psi _{n}(P),G_{n}(P)=\gamma ^{-2n^{2}}\phi
_{n}(P),H_{n}(P)=\gamma ^{-3n^{2}}\omega _{n}(P)  \label{b}
\end{equation}%
are suitably normalized division polynomials of $E$. Note that $F_{0}(P)=0$
and $F_{1}(P)=1$. Furthermore the polynomials $\phi _{n}(P)$ and $\omega
_{n}(P)$ are given by the recursion formulas 
\begin{eqnarray}
\phi _{0}(P) &=&1\text{, }\phi _{1}(P)=x\text{,} \\
\omega _{0}(P) &=&1\text{, }\omega _{1}(P)=y\text{,}  \notag
\end{eqnarray}%
and%
\begin{eqnarray}
\phi _{n}(P) &=&x\psi _{n}(P)^{2}-\psi _{n+1}(P)\psi _{n-1}(P)\text{,}
\label{2d} \\
\omega _{n}(P) &=&(\psi _{n-1}(P)^{2}\psi _{n+2}(P)-\psi _{n-2}(P)\psi
_{n+1}(P)^{2}  \notag \\
&&-\psi _{2}(P)\psi _{n}(P)(a_{1}\phi _{n}(P)+a_{3}\psi _{n}(P)))(2\psi
_{2}(P))^{-1}\text{.}  \notag
\end{eqnarray}%
for all $n\geq 2$. The normalized division polynomials $G_{n}(P)$ and $%
H_{n}(P)$ hold the following relations for some $\gamma \in K^{\ast }$%
\begin{eqnarray}
G_{0}(P) &=&1\text{, }G_{1}(P)={\small \ }\gamma ^{-2}x\text{,}  \label{21}
\\
H_{0}(P) &=&1\text{, }H_{1}(P)={\small \ }\gamma ^{-3}y\text{,}  \label{211}
\end{eqnarray}%
and%
\begin{eqnarray}
G_{n}(P) &=&{\small \ }x\gamma ^{-2}F_{n}(P)^{2}-F_{n+1}(P)F_{n-1}(P)\text{%
{\small ,}}  \label{2b} \\
H_{n}(P) &=&(F_{n-1}(P)^{2}F_{n+2}(P)-F_{n-2}(P)F_{n+1}(P)^{2}  \label{22} \\
&&-\gamma ^{-1}F_{2}(P)F_{n}(P)(a_{1}G_{n}(P)+\gamma
^{-2}a_{3}F_{n}(P)))(2F_{2}(P))^{-1}  \notag
\end{eqnarray}%
for all $n\geq 2$.

Division polynomials play crucial roles in the theory of elliptic functions,
in the theory of elliptic curves \cite{ReneS}, in the theory of elliptic
divisibility sequences, see \cite{MW1}, \cite{MW2}. Ayad \cite{MA} used
explicit addition formulas to prove the sequence of values $%
(F_{n}(P))_{n\geq 0}$ of the division polynomials of an elliptic curve $E$
at a point $P$ is purely periodic modulo prime powers. Cheon and Hahn \cite%
{CH} estimate valuations of division polynomials $F_{n}(P)$. Complete
formulas for explicit valuations of division polynomials at primes of good
or bad reduction are given in \cite{KS}. Silverman \cite{JS3} used
sophisticated methods to study the arithmetic properties of the sequence $%
(F_{n}(P))_{n\geq 0}$. Silverman \cite{JS3} also studied $p$-adic properties
of the sequence $(F_{n}(P))_{n\geq 0}$, and proved the existence and
algebraicity of the $p$-adic limit of certain subsequences of the sequence $%
(F_{n}(P))_{n\geq 0}$. More precisely, Silverman proved if the elliptic
curve $E$ has good reduction, then there is a power $q=p^{e}$ such that for
every $m\geq 1$, the limit%
\begin{equation*}
\lim_{i\rightarrow \infty }F_{mq^{i}}(P)\ \ \text{converges in }%
\mathbb{Z}
_{p}\text{ and is algebraic over }%
\mathbb{Q}
\text{.}
\end{equation*}

The sequences $(G_{n}(P))_{n\geq 0}$ and $(H_{n}(P))_{n\geq 0}$ that are
generated by the numerators of the $x$- and $y$-coordinates of the multiples
of a point $P$ on an elliptic curve $E$ defined over a field $K$ are also
interesting and have properties similar to the sequence $(F_{n}(P))_{n\geq
0} $. In \cite{BO3}, the author and Bizim study periodicity properties and $%
p $-adic properties of the sequences $(G_{n}(P))_{n\geq 0}$ and $%
(H_{n}(P))_{n\geq 0}$. The authors show that the sequences $%
(G_{n}(P))_{n\geq 0}$ and $(H_{n}(P))_{n\geq 0}$ are periodic when $K$ is a
finite field. Moreover, we prove that certain subsequences of these
sequences are converge in $%
\mathbb{Z}
_{p}$ and the limits are algebraic over $%
\mathbb{Q}
$.

In this paper we continue to study the properties of the sequences $%
(G_{n}(P))_{n\geq 0}$ and $(H_{n}(P))_{n\geq 0}$ of values of the suitably
normalized division polynomials of $E$ evaluated at a point $P$ $\in E(K)$.
Let $L$ be a lattice in $%
\mathbb{C}
$, and let $E$ be an elliptic curve defined over $\mathbb{C}$ given by
equation 
\begin{equation*}
E:y^{2}=x^{3}-\frac{1}{4}g_{2}(L)x-\frac{1}{4}g_{3}(L).
\end{equation*}%
Ward \cite[equations 13.6, 13.7]{MW1}, proved that the modular invariants $%
g_{2}(L)$ and $g_{3}(L)$ associated to the lattice $L$ and the Weierstrass
values $\wp (z,L)$ and $\wp ^{\prime }(z,L)$ associated to the point $z$ on
the elliptic curve $%
\mathbb{C}
/L$ are rational functions of $F_{2}$, $F_{3}$, and $F_{4}$, with $%
F_{2}F_{3}\neq 0$, see also \cite[Appendix]{JS4}. Our first main theorem
shows that $g_{2}(L)$, $g_{3}(L)$, $\wp (z,L)$ and $\wp ^{\prime }(z,L)$ are
all defined in the same field as the terms of the sequences $%
(G_{n}(P))_{n\geq 0}$ and $(H_{n}(P))_{n\geq 0}$ similar to that of the
sequence $(F_{n}(P))_{n\geq 0}$. The proof of the theorem uses properties of
elliptic functions.

\begin{theorem}
Let $L$ be a lattice in $%
\mathbb{C}
$, let $E$ be an elliptic curve defined over $\mathbb{C}$ given by equation 
\begin{equation*}
E:y^{2}=x^{3}-\frac{1}{4}g_{2}(L)x-\frac{1}{4}g_{3}(L)
\end{equation*}%
and let $P\in E(\mathbb{C})$. Let $(G_{n}(P))_{n\geq 0}$ and $%
(H_{n}(P))_{n\geq 0}$ be the sequences generated by the numerators of the $x$%
- and $y$-coordinates of the multiples of $P$ as in (\ref{a}), respectively.
Then the modular invariants $g_{2}(L)$ and $g_{3}(L)$ associated to the
lattice $L$ and the Weierstrass values $\wp (z,L)$ and $\wp ^{\prime }(z,L)$
associated to the point $z$ on the elliptic curve $%
\mathbb{C}
/L$ are in the field $\mathbb{Q(}G_{1}$, $G_{2}$, $H_{1}$, $H_{2})$.
\end{theorem}

Section 2 provides background on elliptic divisibility sequences and
elliptic curves. In Section 3, we give a representation of the sequences $%
(G_{n}(P))_{n\geq 0}$ and $(H_{n}(P))_{n\geq 0}$ by means of the elliptic
functions and give the proof of Theorem 1.1. In Section 4 and Appendix A, we
consider the sequences $(G_{n}(P))_{n\geq 0}$ and $(H_{n}(P))_{n\geq 0}$
associated to elliptic curves with a torsion point of order $N$. Ward \cite[%
Theorem 23.1]{MW1} studied the case $N=2$ for elliptic divisibility
sequences. It is a classical result that all elliptic curves with a torsion
point of order $N$ lie in a one parameter family where $N\in \{4,...,10,12\}$%
. In \cite[Theorem 3.2]{BG}, we use Tate normal form of an elliptic curve to
give a complete description of elliptic divisibility sequences arising from
a point of order $N$. In Theorem 4.3, and Appendix A, we give a complete
description of sequences $(G_{n}(P))_{n\geq 0}$ and $(H_{n}(P))_{n\geq 0}$
arising from points of order $N$. We will also use Tate normal form of an
elliptic curve to give the sequences $(G_{n}(P))_{n\geq 0}$ and $%
(H_{n}(P))_{n\geq 0}$ arising from points of order $N$. As an application,
in Theorem 5.1 and Appendix B, we determine square and cube terms in the
sequences $(G_{n}(P))_{n\geq 0}$ and $(H_{n}(P))_{n\geq 0}$ associated to a
Tate normal form.

\textbf{\ Acknowledgements}. This work was supported by the research fund of
Bursa Uluda\u{g} University project no: KUAP(F)-2017/3.

\section{Elliptic Divisibility Sequences}

An \textit{elliptic divisibility sequence} (EDS) is a sequence $%
(h_{n})_{n\geq 0}$ of integers satisfying a recurrence relation of the form%
\begin{equation*}
h_{m+n}h_{m-n}=h_{m+1}h_{m-1}h_{n~}^{2}-h_{n+1}h_{n-1}h_{m~}^{2}
\end{equation*}%
and the divisibility property%
\begin{equation*}
h_{n}|h_{m}\text{ whenever }n|m
\end{equation*}%
for all $m\geq n\geq 1$. An elliptic divisibility sequence is called \textit{%
proper} if $h_{0}=0$, $h_{1}=1$, and $h_{2}h_{3}\neq 0$. The \textit{%
discriminant} of an EDS $(h_{n})_{n\geq 0}$ is the quantity%
\begin{eqnarray*}
\Delta (h_{n})
&=&h_{4}h_{2}^{15}-h_{3}^{3}h_{2}^{12}+3h_{4}^{2}h_{2}^{10}-20h_{4}h_{3}^{3}h_{2}^{7}+3h_{4}^{3}h_{2}^{5}
\\
&&\text{ \ \ \ \ \ \ \ }%
+16h_{3}^{6}h_{2}^{4}+8h_{4}^{2}h_{3}^{3}h_{2}^{2}+h_{4}^{4}\text{,}
\end{eqnarray*}%
(this is the formula in \cite{JS3} or \cite{JS4}, see also \cite{MW1}). A
proper EDS is called \textit{nonsingular} if $\Delta (h_{n})\neq 0$. The
arithmetic properties of EDSs were first studied by Morgan Ward in 1948 \cite%
{MW1, MW2}. For more details on EDSs, see also \cite{GE1, RS, CS}.

Ward defined the division polynomials over the field $%
\mathbb{C}
$ and using the complex analytic theory of elliptic functions showed that
nonsingular elliptic divisibility sequences can be expressed in terms of
elliptic functions. More precisely, Ward \cite[Theorem 12.1]{MW1} proved
that if $(h_{n})_{n\geq 0}$ is a nonsingular elliptic divisibility sequence,
then there exist a lattice $L\subset 
\mathbb{C}
$ and a complex number $z$ $\in 
\mathbb{C}
$ such that%
\begin{equation}
h_{n}=\psi _{n}(z,L)=\frac{\sigma (nz\text{, }L)}{\sigma (z\text{, }%
L)^{n^{2}}}\text{ for all }n\geq 1\text{,}  \label{aa1}
\end{equation}%
where $\psi _{n}(z$, $L)$ and $\sigma (z$, $L)$ are the $n$-division
polynomial and the Weierstrass $\sigma $-function associated to the lattice $%
L$, respectively. Further, Ward showed the modular invariants $g_{2}(L)$ and 
$g_{3}(L)$ associated to the lattice $L$ and the Weierstrass values $\wp (z)$
and $\wp ^{\prime }(z)$ associated to the point $z$ on the elliptic curve $%
\mathbb{C}
/L$ can be given by the terms $h_{2}$, $h_{3}$ and $h_{4}$ of the sequence $%
(h_{n})$, see \cite[equations 13.6, 13.7, 13.5 and 13.1]{MW1}. Silverman 
\cite[Proposition 18]{JS3} reformulated Ward's result and showed that if $%
(h_{n})_{n\geq 0}$ is a nonsingular EDS associated to an elliptic curve $E$
given by a minimal Weierstrass equation over $%
\mathbb{Q}
$ and a point $P\in E(%
\mathbb{Q}
)$, then there is a constant $\gamma \in 
\mathbb{Q}
^{\ast }$ such that 
\begin{equation}
h_{n}=\gamma ^{n^{2}-1}F_{n}(P)\text{ \ \ \ for all }n\geq 1  \label{a1}
\end{equation}%
where $F_{n}$ is the normalized $n$-division polynomial on $E$.

\section{The Representation of the Sequences $(G_{n}(P))_{n\geq 0}$ and $%
(H_{n}(P))_{n\geq 0}$ by Elliptic Functions}

Let $E$ be an elliptic curve defined over a field $K$ with Weierstrass
equation%
\begin{equation}
E:y^{2}=x^{3}+ax+b\text{. }
\end{equation}%
It is clear that the coefficients of the elliptic curve $E$ can be defined
in terms of the sequence $(F_{n}(P))_{n\geq 0}$ of values of the division
polynomials of $E$ at a point $P$ by using the relation (\ref{a1}) and
Ward's formulas for\ the modular invariants $g_{2}(L)$ and $g_{3}(L)$ \cite[%
equations 13.6, 13.7]{MW1}; see also \cite[Appendix]{JS4}. In this section
we give a representation of the sequences $(G_{n}(P))_{n\geq 0}$ and $%
(H_{n}(P))_{n\geq 0}$ of values of the suitably normalized division
polynomials of $E$ evaluated at a point $P$ $\in E(K)$ by means of the
elliptic functions and prove the coefficients of the elliptic curve $E$ can
be defined in terms of these sequences. In this section we will also assume $%
\psi _{2}(P)\psi _{3}(P)\neq 0$ so that $F_{2}(P)F_{3}(P)\neq 0$.

We first state some results from elliptic function theory that will be
needed. Let $L$ be a lattice in $%
\mathbb{C}
$. Recall from elliptic function theory that the Weierstrass $\wp $-function
associated to the lattice $L$ and its derivative $\wp ^{\prime }$ satisfy 
\begin{equation}
\wp ^{\prime }(z)^{2}=4\wp (z)^{3}-g_{2}(L)\wp (z)-g_{3}(L)\text{ }
\label{1}
\end{equation}%
where $g_{2}(L)$ and $g_{3}(L)$ are modular invariants associated to the
lattice $L$. If we take the derivative of the both sides of (\ref{1}) we
have the following relation%
\begin{equation}
\wp ^{\prime \prime }\newline
(z)=6\wp (z)^{2}-\frac{1}{2}g_{2}(L)\text{. }  \label{3}
\end{equation}%
Then we obtain%
\begin{equation}
g_{2}(L)=12\wp (z)^{2}-2\wp ^{\prime \prime }\newline
(z)\text{. }  \label{3a}
\end{equation}%
by (\ref{3}) and so%
\begin{equation}
g_{3}(L)=2\wp (z)[\wp ^{\prime \prime }\newline
(z)-4\wp (z)^{2}]-\wp ^{\prime }(z)^{2}  \label{3b}
\end{equation}%
by (\ref{1}) and (\ref{3}). Moreover recall that%
\begin{equation}
\psi _{2}(z)=-\wp ^{\prime }(z)\text{,}  \label{2a}
\end{equation}%
\begin{equation}
\psi _{3}(z)=3\wp (z)^{4}-\frac{3}{2}g_{2}(L)\wp (z)^{2}-3g_{3}(L)\wp (z)-%
\frac{1}{16}g_{2}(L)^{2}\text{,}
\end{equation}%
and%
\begin{equation}
\wp (2z)-\wp (z)=\frac{1}{4}\left( \frac{\wp ^{\prime \prime }\newline
(z)}{\wp ^{\prime }(z)}\right) ^{2}-3\wp (z)\text{,}
\end{equation}%
\begin{equation}
\wp (3z)-\wp (z)=\frac{\wp ^{\prime }(z)^{2}[\wp ^{\prime }(z)^{4}-\psi
_{3}(z)\wp ^{\prime \prime }\newline
(z)]}{\psi _{3}(z)^{2}}\text{.}  \label{4}
\end{equation}%
Furthermore, one can derive a formula for $\wp (nz)$ in terms of $\wp (z)$, $%
\psi _{n}(z)$ and $\psi _{n\pm 1}(z)$, more explicitly the following
relation holds%
\begin{equation}
\wp (nz)=\wp (z)-\frac{\psi _{n+1}(z)\psi _{n-1}(z)}{\psi _{n}(z)^{2}}
\label{4a}
\end{equation}%
for all $n$ $\geq 2$. Thus substituting $n=2$ and $n=3$ into (\ref{4a}) we
have 
\begin{equation}
\wp (2z)-\wp (z)=-\frac{\psi _{3}(z)}{\psi _{2}(z)^{2}}\text{ \ \ \ }
\label{4b}
\end{equation}%
since $\psi _{1}(z)=1$, and 
\begin{equation}
\wp (3z)-\wp (z)=-\frac{\psi _{4}(z)\psi _{2}(z)}{\psi _{3}(z)^{2}}\text{.}
\label{4c}
\end{equation}%
Now by (\ref{4}), (\ref{2a}) and (\ref{4c}) we have 
\begin{equation*}
-\frac{\psi _{4}(z)\psi _{2}(z)}{\psi _{3}(z)^{2}}=\frac{\psi
_{2}(z)^{2}[\psi _{2}(z)^{4}-\psi _{3}(z)\wp ^{\prime \prime }\newline
(z)]}{\psi _{3}(z)^{2}}
\end{equation*}%
and so%
\begin{equation}
\wp ^{\prime \prime }(z)=\frac{\psi _{2}(z)^{5}+\psi _{4}(z)}{\psi
_{2}(z)\psi _{3}(z)}\text{.}  \label{4d}
\end{equation}%
Let $E$ be an elliptic curve over $\mathbb{C}$. Then the points $\left( \wp
(z),\wp ^{\prime }(z)\right) $ lie on the elliptic curve 
\begin{equation}
y^{2}=4x^{3}-g_{2}(L)x-g_{3}(L)  \label{1a}
\end{equation}%
by (\ref{1}). Now let $(F_{n}(P))_{n\geq 0}$, $(G_{n}(P))_{n\geq 0}$ and $%
(H_{n}(P))_{n\geq 0}$ be the sequences of values of the normalized division
polynomials of $E$ at a point $P$. Then by second part of (\ref{21}) we have%
\begin{equation}
\wp (z)=\gamma ^{2}G_{1}(P)\text{.}  \label{c}
\end{equation}%
On the other hand by (\ref{2b}),%
\begin{equation}
G_{2}(P)=x\gamma ^{-2}F_{2}(P)^{2}-F_{3}(P)  \label{c0}
\end{equation}%
since $F_{1}(P)=1$. By first part of (\ref{b}), and (\ref{4b}) we obtain%
\begin{equation}
\wp (2z)=\frac{\wp (z)F_{2}(P)^{2}-\gamma ^{2}F_{3}(P)}{F_{2}(P)^{2}}\text{.}
\label{c2}
\end{equation}%
Hence by second part of (\ref{21}), (\ref{c}) and (\ref{c0}) we have 
\begin{equation*}
\wp (2z)=\frac{\gamma ^{2}G_{2}(P)}{F_{2}(P)^{2}}\text{.}
\end{equation*}%
Thus one can easily derive inductively that%
\begin{equation}
\wp (nz)=\frac{\gamma ^{2}G_{n}(P)}{F_{n}(P)^{2}}  \label{2}
\end{equation}%
for all $n\geq 1$.

We are now ready to prove our first main result. From now on, for simplicity
of notation, we write $G_{n}$ and $H_{n}$ for $G_{n}(P)$ and $H_{n}(P)$,
respectively, unless otherwise specified.

\begin{proof}[Proof of Theorem 1.1]
By the first part of\textit{\ }(\ref{b}) we have 
\begin{equation}
\psi _{2}(z)=\gamma ^{3}F_{2}  \label{d0}
\end{equation}%
and 
\begin{equation}
\psi _{3}(z)=\gamma ^{8}F_{3}\text{.}  \label{d1}
\end{equation}%
Thus by (\ref{d0}) and (\ref{2a}) we obtain 
\begin{equation}
\wp ^{\prime }(z)=-\gamma ^{3}F_{2}\text{.}  \label{c1}
\end{equation}%
On the other hand (\ref{2}) implies that 
\begin{equation}
\wp (3z)=\gamma ^{2}G_{3}/F_{3}^{2}\text{.}  \label{d}
\end{equation}%
Now (\ref{4d}) and the first part of (\ref{b}) imply that%
\begin{equation}
\wp ^{\prime \prime }(z)=\frac{\gamma ^{4}(F_{2}^{5}+F_{4})}{F_{2}F_{3}}%
\text{.}  \label{d2}
\end{equation}%
On the other hand by (\ref{d0}) we have%
\begin{equation*}
F_{2}=2\gamma ^{-3}y
\end{equation*}%
since $\psi _{2}=2y$, for the elliptic curve $E:y^{2}=x^{3}-\frac{1}{4}%
g_{2}(L)x-\frac{1}{4}g_{3}(L)$. Thus by the second part of (\ref{211}) we
derive that 
\begin{equation}
F_{2}=2H_{1}\text{.}  \label{f0}
\end{equation}%
Now by putting $n=2$ into (\ref{22}) and then using (\ref{f0}) we obtain 
\begin{equation}
F_{4}=4H_{1}H_{2}  \label{f2}
\end{equation}%
since $F_{0}=0$ and $F_{1}=1$. Thus 
\begin{equation}
\wp ^{\prime }(z)=-2\gamma ^{3}H_{1}\text{,}  \label{f}
\end{equation}%
by (\ref{c1}) and (\ref{f0}). Now by setting $n=2$ in (\ref{2b}) and then
using (\ref{c1}) and the second part of (\ref{21}) we have 
\begin{equation}
F_{3}=\gamma ^{-6}\wp ^{\prime }(z)^{2}G_{1}-G_{2}  \label{e2}
\end{equation}%
since $F_{1}=1$. Thus 
\begin{equation}
F_{3}=4G_{1}H_{1}^{2}-G_{2}  \label{e3}
\end{equation}%
by (\ref{f}). Therefore by (\ref{d2}), (\ref{f0}), (\ref{f2}) and (\ref{e3})
we have%
\begin{equation}
\wp ^{\prime \prime }(z)=\frac{2\gamma ^{4}(8H_{1}^{4}+H_{2})}{%
4G_{1}H_{1}^{2}-G_{2}}\text{.}  \label{e1}
\end{equation}

On combining (\ref{3a}) with (\ref{c}) and (\ref{e1}) we obtain the
following formula for $g_{2}(L)$, 
\begin{equation}
g_{2}(L)=\frac{4\gamma
^{4}(12G_{1}^{3}H_{1}^{2}-3G_{1}^{2}G_{2}-8H_{1}^{4}-H_{2})}{%
4G_{1}H_{1}^{2}-G_{2}}\text{.}  \label{g}
\end{equation}%
Similarly combining (\ref{3b}) with (\ref{c}), (\ref{f}) and (\ref{e1}) we
have%
\begin{equation}
g_{3}(L)=\frac{4\gamma
^{6}(4G_{1}H_{1}^{4}+G_{1}H_{2}-8G_{1}^{4}H_{1}^{2}+2G_{1}^{3}G_{2}+H_{1}^{2}G_{2})%
}{4G_{1}H_{1}^{2}-G_{2}}\text{. }  \label{h1}
\end{equation}%
Now if $E$ is an elliptic curve over $\mathbb{Q}$ given by a Weierstrass
equation%
\begin{equation*}
E:y^{2}=x^{3}+ax+b\text{,}
\end{equation*}%
then (\ref{1a}) imply that%
\begin{equation*}
a=-\frac{g_{2}(L)}{4}\text{ and }b=-\frac{g_{3}(L)}{4}
\end{equation*}%
where $g_{2}(L)$ and $g_{3}(L)$ are the rational expressions in $G_{1}$, $%
G_{2}$, $H_{1}$ and $H_{2}$ by relations (\ref{g}) and (\ref{h1})
respectively. Finally rational expressions for $\wp (z,L)$ and $\wp ^{\prime
}(z,L)$ are given by equations (\ref{c}) and (\ref{f}), which completes the
proof of the theorem.
\end{proof}

\section{The Sequences $(G_{n})_{n\geq 0}$ and $(H_{n})_{n\geq 0}$
Associated to Tate Normal Forms}

The study of the group $E(%
\mathbb{Q}
)$ has been playing important roles in number theory. The modern number
theory originated in 1922 when L. J. Mordell proved that the group of
rational points $E(%
\mathbb{Q}
)$ is a finitely generated abelian group. This result was generalized in
1928 to abelian varieties over number fields by A. Weil. Moreover, the
characterization of torsion subgroups of $E(%
\mathbb{Q}
)$ is always interesting. A uniform bound was studied for the order of the
torsion subgroup $E_{tors}(%
\mathbb{Q}
{\mathbb{)}}$ of $E(%
\mathbb{Q}
)$ by Shimura, Ogg, and others. The following result conjectured by Ogg, was
proved by B. Mazur.

\begin{theorem}[\protect\cite{BM}]
Let $E$ be an elliptic curve defined over $%
\mathbb{Q}
$. Then the torsion subgroup $E_{tors}(%
\mathbb{Q}
{\mathbb{)}}$ is either isomorphic to $%
\mathbb{Z}
/N%
\mathbb{Z}
$ for $N=1,2,...,10,12$ or to $%
\mathbb{Z}
/2%
\mathbb{Z}
\times 
\mathbb{Z}
/2N%
\mathbb{Z}
$ for $N=1,2,3,4$. Further, each of these groups does occur as an $E_{tors}(%
\mathbb{Q}
{\mathbb{)}}$.
\end{theorem}

It is a classical result that all elliptic curves with a torsion point of
order $N$ lie in a one parameter family where $N\in \{4,...,10,12\}$. The 
\textit{Tate normal form} of an elliptic curve $E$ with point $P=(0,0)$ is
given by 
\begin{equation*}
E_{N}:y^{2}+(1-c)xy-by=x^{3}-bx^{2}
\end{equation*}%
where the point $P$ has given order $N$.

If an elliptic curve in normal form has a point of order $N>3$, then
admissible change of variables transforms the curve to the Tate normal form,
in this case the point $P=(0,0)$ is a torsion point of maximal order. Kubert 
\cite{KU} gives a list of parameterizable torsion structures, which includes
one parameter family of elliptic curves $E$ defined over $\mathbb{Q}$ with a
torsion point of order $N$ where $N=4,...,10$, $12$. Some algorithms are
given by using the existence of such a family, see \cite{GO} for more
details. In order to describe when an elliptic curve defined over $\mathbb{Q}
$ has a point of given order $N$, we need the following result on
parametrization of torsion structures. Most cases of the following
parameterizations are proved by Husem\"{o}ller \cite{HS}.

\begin{theorem}[\protect\cite{GO}]
Every elliptic curve with point $P=(0$, $0)$ of order $N=4$, $...$, $10$, $%
12 $ can be written in the following Tate normal form%
\begin{equation*}
E_{N}:y^{2}+(1-c)xy-by=x^{3}-bx^{2},
\end{equation*}%
with the following relations:\newline
1. If $N=4$, then $b=\alpha $ and $c=0$, $\alpha \neq 0$.\newline
2. If $N=5$, then $b=\alpha $ and $c=\alpha $, $\alpha \neq 0$.\newline
3. If $N=6$, then $b=\alpha +\alpha ^{2}~$and $c=\alpha $, $\alpha \neq -1$, 
$0$.\newline
4. If $N=7$, then $b=\alpha ^{3}-\alpha ^{2}~$and $c=\alpha ^{2}-\alpha $, $%
\alpha \neq 0$, $1$.\newline
5. If $N=8$, then $b=(2\alpha -1)(\alpha -1)$ and$~c=b/\alpha $, $\alpha
\neq 0$, $\frac{1}{2}$, $1$.\newline
6. If $N=9$, then $c=\alpha ^{2}(\alpha -1)~$and $b=c(\alpha (\alpha -1)+1)$%
, $\alpha \neq 0$, $1$.\newline
7. If $N=10$, then $c=(2\alpha ^{3}-3\alpha ^{2}+\alpha )/(\alpha -(\alpha
-1)^{2})$ and$~b=c\alpha ^{2}/(\alpha -(\alpha -1)^{2})$, $\alpha \neq 0$, $%
\frac{1}{2}$, $1$.\newline
8. If $N=12$, then $c=(3\alpha ^{2}-3\alpha +1)(\alpha -2\alpha
^{2})/(\alpha -1)^{3}$ and$~b=c(-2\alpha ^{2}+2\alpha -1)/(\alpha -1)$, $%
\alpha \neq 0$, $\frac{1}{2}$, $1$.
\end{theorem}

Theorem 4.2 says that every elliptic curve with a point of order $N$ is
birationally equivalent to one of the Tate normal forms given above. We will
assume that the parameter $\alpha \in 
\mathbb{Z}
$ and the coefficients of $E_{N}$ are chosen to lie in $%
\mathbb{Z}
$. Hence for $N=8$, $10$, $12$, we transform $E_{N}$ into a birationally
equivalent curve $E_{N}^{\prime }$ having an equation with integral
coefficients. The equations of the birationally equivalent curves for $N=8$, 
$10$, $12$ are given, respectively, as follows:%
\begin{eqnarray*}
{\small E}_{8}^{\prime } &\text{:}&{\small y}^{2}{\small +(\alpha -\beta
)xy-\alpha }^{3}{\small \beta y=x}^{3}{\small -\alpha }^{2}{\small \beta x}%
^{2}\text{,\newline
} \\
{\small E}_{10}^{\prime } &\text{:}&{\small y}^{2}{\small +(\zeta }^{2}-%
{\small \alpha \beta \zeta )xy-\alpha }^{3}{\small \beta \zeta }^{4}{\small %
y=x}^{3}{\small -\alpha }^{3}{\small \beta \zeta }^{2}{\small x}^{2}\text{,}
\\
{\small E}_{12}^{\prime } &\text{:}&{\small y}^{2}{\small +(\alpha
-1)((\alpha -1)}^{3}{\small -\lambda )xy-(\alpha -1)}^{8}{\small \lambda
\theta y=x}^{3}{\small -(\alpha -1)}^{4}{\small \lambda \theta x}^{2}\text{,}
\end{eqnarray*}%
where $\alpha \neq 0$, $1$, 
\begin{equation}
\beta =(2\alpha -1)(\alpha -1)\text{, }\zeta =-\alpha ^{2}+3\alpha -1\text{, 
}  \label{be}
\end{equation}%
and 
\begin{equation}
\lambda =(3\alpha ^{2}-3\alpha +1)(\alpha -2\alpha ^{2})\text{, }\theta
=2\alpha -2\alpha ^{2}-1\text{.}  \label{la}
\end{equation}%
From now on, for simplicity of notation, we write $E_{8}$, $E_{10}$, $E_{12}$
for $E_{8}^{\prime }$, $E_{10}^{\prime }$, $E_{12}^{\prime }$, respectively.

In \cite[Theorem 3.2]{BG}, we give the general terms of the elliptic
divisibility sequences $(h_{n})_{n\geq 0}$ associated to a Tate normal form $%
E_{N}$ of an elliptic curve for some integer parameter $\alpha $. For
example, the general term of $(h_{n})_{n\geq 0}$ for $N=8$ is%
\begin{equation}
h_{n}=\varepsilon \alpha ^{\{(15n^{2}-p)/16\}}(\alpha
-1)^{\{(7n^{2}-q)/16\}}(2\alpha -1)^{\{(3n^{2}-r)/8\}}  \label{n8}
\end{equation}%
where 
\begin{equation*}
\varepsilon =\left\{ 
\begin{array}{ll}
{\small +1}\text{,} & \text{if }{\small n}\text{ }{\small \equiv
1,4,5,9,10,13,14~(16)}\text{ \ \ }{\small \ } \\ 
{\small -1}\text{,} & \text{if }{\small n}\text{ }{\small \equiv
2,3,6,7,11,12,15~(16)}\text{,}%
\end{array}%
\right.
\end{equation*}%
and 
\begin{equation*}
{\small p=}\left\{ 
\begin{array}{ll}
{\small 15}\text{,} & \text{if }{\small n}\text{ }{\small \equiv 1,7~(8)} \\ 
{\small 12}\text{,} & \text{if }{\small n}\text{ }{\small \equiv 2,6~(8)} \\ 
{\small 7}\text{,} & \text{if }{\small n}\text{ }{\small \equiv 3,5~(8)} \\ 
{\small 16}\text{,} & \text{if }{\small n}\text{ }{\small \equiv 4~(8),}%
\end{array}%
\right. {\small q=}\left\{ 
\begin{array}{ll}
{\small 7}\text{,} & \text{if }{\small n}\text{ }{\small \equiv 1,7~(8)} \\ 
{\small 12}\text{,} & \text{if }{\small n}\text{ }{\small \equiv 2,6~(8)} \\ 
{\small 15}\text{,} & \text{if }{\small n}\text{ }{\small \equiv 3,5~(8)} \\ 
{\small 16}\text{,} & \text{if }{\small n}\text{ }{\small \equiv 4~(8),}%
\end{array}%
\right. r~{\small =}\left\{ 
\begin{array}{ll}
{\small 3}\text{,} & \text{if }{\small n}\text{ }{\small \equiv 1,3,5,7~(8)}
\\ 
{\small 4,} & \text{if }{\small n}\text{ }{\small \equiv 2,6~(8)} \\ 
{\small 0}\text{,} & \text{if }{\small n}\text{ }{\small \equiv 4~(8)}\text{%
{\small .}}%
\end{array}%
\right.
\end{equation*}

In Section 2, we recall that if $(h_{n})_{n\geq 0}$ is a nonsingular EDS
associated to an elliptic curve $E$ given by a minimal Weierstrass equation
over $%
\mathbb{Q}
$ and a point $P\in E(%
\mathbb{Q}
)$, then there is a constant $\gamma \in 
\mathbb{Q}
^{\ast }$ such that 
\begin{equation*}
h_{n}=\gamma ^{n^{2}-1}F_{n}(P)\text{ \ \ \ for all }n\geq 0\text{.}
\end{equation*}%
Therefore, one can easily obtain the general term of the sequence $%
(F_{n})_{n\geq 0}$ associated to a Tate normal form $E_{N}$, by using the
relation above.

In this section we consider $(G_{n})_{n\geq 0}$ and $(H_{n})_{n\geq 0}$
sequences associated to a Tate normal form $E_{N}$ with torsion point $%
P=(0,0)$ and give the general terms of these sequences. We take $\gamma =1$
in (\ref{b}) so that $G_{n}=\phi _{n}$, $H_{n}=\omega _{n}$, and 
\begin{equation}
F_{n}=h_{n}\text{ \ \ \ for all }n\geq 0  \label{a2}
\end{equation}%
by (\ref{a1}).

In the following theorem we determine general terms of the sequences $%
(G_{n})_{n\geq 0}$ and $(H_{n})_{n\geq 0}$ associated to an elliptic curve
in Tate normal form with a torsion point $P=(0,0)$ of order $8$. For the
convenience of the reader, we have given the other cases in Appendix A. The
proof uses the general terms of sequences in \cite[Theorem 3.2]{BG}.

\begin{theorem}
Let $E_{8}$ be a Tate normal form of an elliptic curve with a torsion point $%
P=(0,0)$ of order $8$. Let $(G_{n})_{n\geq 0}$ and $(H_{n})_{n\geq 0}$ be
the sequences generated by the numerators of the $x$- and $y$-coordinates of
the multiples of $P$ as in (\ref{a}). Then the general terms of the
sequences $(G_{n})_{n\geq 0}$ and $(H_{n})_{n\geq 0}$ can be given by the
following formulas:\newline
\begin{equation}
{\small G}_{n}~{\small =}\left\{ 
\begin{array}{cc}
{\small 0}\text{,} & \text{if }{\small n}\text{ }{\small \equiv 1,7~(8)} \\ 
{\small \alpha }^{\{(15n^{2}+a_{1})/8\}}{\small (\alpha -1)}%
^{\{(7n^{2}-b_{1})/8\}}{\small (2\alpha -1)}^{\{(3n^{2}+c_{1})/4\}}\text{,}
& \text{otherwise,}%
\end{array}%
\right.  \label{13}
\end{equation}%
and%
\begin{equation}
{\small H}_{n}~{\small =}\left\{ 
\begin{array}{cc}
{\small 0}\text{,} & \text{if }{\small n}\text{ }{\small \equiv 1,6~(8)} \\ 
{\small \varepsilon \alpha }^{\{(45n^{2}+a_{2})/16\}}{\small (\alpha -1)}%
^{\{(21n^{2}-b_{2})/16\}}{\small (2\alpha -1)}^{\{(9n^{2}-c_{2})/8\}}\text{,}
& \text{otherwise,}%
\end{array}%
\right.
\end{equation}%
where $\alpha \neq 0$, $1$, 
\begin{equation*}
a_{1}=\left\{ 
\begin{array}{ll}
0\text{,} & \text{if }n\equiv 0~(8) \\ 
4\text{,} & \text{if }n\equiv 2,6~(8) \\ 
1\text{,} & \text{if }n\equiv 3,5~(8) \\ 
8\text{,} & \text{if }n\equiv 4~(8)\text{,}%
\end{array}%
\right. b_{1}=\left\{ 
\begin{array}{ll}
0\text{,} & \text{if }n\equiv 0~(8) \\ 
4\text{,} & \text{if }n\equiv 2,6~(8) \\ 
7\text{,} & \text{if }n\equiv 3,5~(8) \\ 
8\text{,} & \text{if }n\equiv 4~(8)\text{,}%
\end{array}%
\right. c_{1}=\left\{ 
\begin{array}{ll}
0\text{,} & \text{if }n\equiv 0,2,4,6~(8) \\ 
1\text{,} & \text{if }n\equiv 3,5~(8)\text{.}%
\end{array}%
\right.
\end{equation*}%
and%
\begin{equation*}
\varepsilon =\left\{ 
\begin{array}{ll}
+1\text{,} & \text{if }n\equiv 0,4,5,10,13~(16)\text{ \ \ }\  \\ 
-1\text{,} & \text{if }n\equiv 2,3,7,8,11,12,15~(16)\text{,}%
\end{array}%
\right.
\end{equation*}%
\begin{equation*}
a_{2}=\left\{ 
\begin{array}{ll}
{\small 0}\text{,} & \text{if }{\small n}\text{ }{\small \equiv 0~(8)} \\ 
-{\small 4}\text{,} & \text{if }{\small n}\text{ }{\small \equiv 2~(8)} \\ 
{\small 11}\text{,} & \text{if }{\small n}\text{ }{\small \equiv 3~(8)} \\ 
{\small 16}\text{,} & \text{if }{\small n}\text{ }{\small \equiv 4~(8)} \\ 
-{\small 5}\text{,} & \text{if }{\small n}\text{ }{\small \equiv 5~(8)} \\ 
{\small 3}\text{,} & \text{if }{\small n}\text{ }{\small \equiv 7~(8)}\text{%
{\small ,}}%
\end{array}%
\right. b_{2}=\left\{ 
\begin{array}{ll}
{\small 0}\text{,} & \text{if }{\small n}\text{ }{\small \equiv 0~(8)} \\ 
{\small 4}\text{,} & \text{if }{\small n}\text{ }{\small \equiv 2~(8)} \\ 
{\small 13}\text{,} & \text{if }{\small n}\text{ }{\small \equiv 3,5~(8)} \\ 
{\small 16}\text{,} & \text{if }{\small n}\text{ }{\small \equiv 4~(8)} \\ 
{\small 5}\text{,} & \text{if }{\small n}\text{ }{\small \equiv 7~(8)}\text{%
{\small ,}}%
\end{array}%
\right. ~c_{2}=\left\{ 
\begin{array}{ll}
{\small 0}\text{,} & \text{if }{\small n}\text{ }{\small \equiv 0,4~(8)} \\ 
{\small 4}\text{{\small ,}} & \text{if }{\small n}\text{ }{\small \equiv
2~(8)} \\ 
{\small 1}\text{,} & \text{if }{\small n}\text{ }{\small \equiv 3,7~(8)} \\ 
{\small -7}\text{,} & \text{if }{\small n}\text{ }{\small \equiv 5~(8)}\text{%
.}%
\end{array}%
\right.
\end{equation*}
\end{theorem}

\begin{proof}
We give the proof only for the sequence $(G_{n})_{n\geq 0}$ as the proof for 
$(H_{n})_{n\geq 0}$ is similar.

Let $n\equiv 1$ $(8)$. Then by (\ref{2b})%
\begin{equation*}
{\small G}_{8k+1}~{\small =-F_{8k+2}F}_{8k}\text{ \ for all }k\geq 0\text{,}
\end{equation*}%
since $x=0$. We note that ${\small F}_{n}~{\small =0}$ if and only if the
order $N$ of the point $P$ divides $n$. It follows that ${\small F}_{8k}~%
{\small =0}$ for all $k\geq 0$, hence ${\small G}_{8k+1}~{\small =0}$.

Now let $n\equiv 2$ $(8)$. Then by (\ref{2b})%
\begin{equation}
{\small G}_{8k+2}~{\small =-F_{8k+3}F}_{8k+1}\text{ for all }k\geq 0\text{.}
\label{2bb}
\end{equation}%
By (\ref{n8}) and (\ref{a2}) we obtain%
\begin{equation*}
{\small F}_{8k+1}~{\small =}{\small \alpha }^{60k^{2}+15k}{\small (\alpha -1)%
}^{28k^{2}+7k}{\small (2\alpha -1)}^{24k^{2}+6k}\text{ \ for all }k\geq 0%
\text{,}
\end{equation*}%
and 
\begin{equation*}
F_{8k+3}~{\small =-}{\small \alpha }^{60k^{2}+45k+8}{\small (\alpha -1)}%
^{28k^{2}+21k+3}{\small (2\alpha -1)}^{24k^{2}+18k+3}\text{ for all }k\geq 0%
\text{.}
\end{equation*}%
Now substituting these expressions into (\ref{2bb}) we derive that 
\begin{equation*}
{\small G}_{8k+2}~{\small =\alpha }^{120k^{2}+60k+8}{\small (\alpha -1)}%
^{56k^{2}+28k+3}{\small (2\alpha -1)}^{48k^{2}+24k+3}\text{.}
\end{equation*}%
for all $k\geq 0$. On the other hand by the general term formula in (\ref{13}%
) we have%
\begin{equation*}
{\small G}_{8k+2}~{\small =\alpha }^{120k^{2}+60k+8}{\small (\alpha -1)}%
^{56k^{2}+28k+3}{\small (2\alpha -1)}^{48k^{2}+24k+3}\text{,}
\end{equation*}%
which completes the proof for $n\equiv 2$ $(8)$. The remaining cases can be
proved in a similar manner.
\end{proof}

\begin{remark}
There is no Tate normal form of an elliptic curve with the torsion point of
order two or three, but in \cite{KU}, Kubert gives a list of elliptic curves
with torsion point of order two or three are 
\begin{equation}
E_{2}:y^{2}=x^{3}+a_{2}x^{2}+a_{4}x\text{, \ }a_{4}\neq 0\text{,}  \label{30}
\end{equation}%
and%
\begin{equation}
E_{3}:y^{2}+a_{1}xy+a_{3}y=x^{3}\text{, \ }a_{3}\neq 0\text{,}  \label{31}
\end{equation}%
respectively.
\end{remark}

The following theorem gives the general term of the sequence $(G_{n})_{n\geq
0}$ associated to $E_{2}$ and $E_{3}$, respectively and the general term of
the sequence $(H_{n})_{n\geq 0}$ associated to elliptic curve $E_{3}$. We
note the sequence $(H_{n})_{n\geq 0}$ associated to elliptic curve $E_{2}$
is not defined since $F_{2}=0$; see relation (\ref{22}). The proof of the
theorem is similar to the proof of Theorem 4.3.

\begin{theorem}
Let $E_{N}$ be an elliptic curve with the torsion point $P=(0,0)$ of order $%
N $ as in (\ref{30}) and (\ref{31}). Let $(G_{n})_{n\geq 0}$ and $%
(H_{n})_{n\geq 0}$ be the sequences generated by the numerators of the $x$-
and $y$-coordinates of the multiples of $P$ as in (\ref{a}). Then the
general term of the sequences $(G_{n})_{n\geq 0}$ and $(H_{n})_{n\geq 0}$
can be given by the following formulas:\newline
1. If $N=2$, then 
\begin{equation*}
G_{n}\ =\left\{ 
\begin{array}{cc}
0\text{,} & \text{if }n\text{ is odd} \\ 
a_{4}^{\{n^{2}/2\}}\text{,} & \text{if }n\text{ is even.}%
\end{array}%
\right.
\end{equation*}%
2. If $N=3$, then 
\begin{equation*}
G_{n}=\left\{ 
\begin{array}{cc}
0\text{,} & \text{if }n\equiv 1,2\text{ }(3) \\ 
a_{3}^{\{2n^{2}/3\}}\text{,} & \text{if }n\equiv 0\text{ }(3)\text{. }%
\end{array}%
\right.
\end{equation*}%
and%
\begin{equation*}
H_{n}=\left\{ 
\begin{array}{cc}
0\text{,} & \text{if }n\equiv 1\text{ }(3) \\ 
\varepsilon a_{3}^{n^{2}}\text{,} & \text{if }n\equiv 0,2\text{ }(3)%
\end{array}%
\right.
\end{equation*}%
\newline
where%
\begin{equation*}
{\small \varepsilon =}\left\{ 
\begin{array}{ll}
+1\text{,} & \text{if }n\equiv 0,5~(6) \\ 
-1\text{,} & \text{if }n\equiv 2,3~(6)\text{.}%
\end{array}%
\right.
\end{equation*}
\end{theorem}

\section{Squares and Cubes in $(G_{n})_{n\geq 0}$ and $(H_{n})_{n\geq 0}$%
Sequences}

The problem of determining square and cube terms in linear sequences has
been\ considered by various authors, see \cite{AP}, \cite{PR}, \cite{PR1}, 
\cite{PR2}, and see also, \cite{AB1}, \cite{AB2}. Similar problem has also
been considered for non-linear sequences, see \cite{BO1}, \cite{BO2}, \cite%
{BG}, see also \cite{R}, \cite{VM}. In this section we determine square and
cube terms in the sequences $(G_{n})_{n\geq 0}$ and $(H_{n})_{n\geq 0}$
associated to a Tate normal form $E_{N}$ of an elliptic curve with a torsion
point $P=(0,0)$ of order $N$. Throughout this paper the symbol $\square $
means a square of a non-zero integer, i.e., $\square =\pm \beta ^{2}$ where $%
\beta $ is a non-zero integer, and $C$ means a cube of a non-zero integer.
Determining square and cube terms in these sequences leads to some equations
and these equations are similar to equations in \cite[Table 3]{BG}.
Therefore we use similar techniques in \cite{BG} for determining square and
cube terms in these sequences. We observe that the irreducible factors
appearing in the left hand side (if they are at least two) of these
equations are pairwise relatively prime (for example, one can easily show
that the irreducible factors in the 
\begin{equation*}
\alpha (\alpha -1)(2\alpha -1)(2\alpha ^{2}-2\alpha +1)(3\alpha ^{2}-3\alpha
+1)=\square ,
\end{equation*}%
are pairwise relatively prime, see \cite{BG}, p. 498). It follows that, if
the right hand side of the equation is $\square $ (or $C$),\ then every
irreducible factor is $\square $ (or $C$). It turns out that it is not
necessary to consider all irreducible factors in the left hand side. For
example, the equation%
\begin{equation*}
\alpha (2\alpha -1)(2\alpha ^{2}-2\alpha +1)=C
\end{equation*}%
implies that all three $\alpha $, $2\alpha -1$, and $2\alpha ^{2}-2\alpha +1$
are $C$, we only use the fact that the third one is $C$.

We use the tables in \cite{AP2} when our equations turned into Mordell's
equation. In some cases we will apply \textit{Elliptic Logarithm Method} to
find all integral solutions of our equations (this method has been developed
in \cite{STR} and, independently, in \cite{GEB} and now is implemented in
MAGMA \cite{MAG}; see also \cite{WBOS}){\small .}

Theorem 5.1 answers the following three questions:

(1) Which terms of the sequence $(G_{n})_{n\geq 0}$ (or $(H_{n})_{n\geq 0}$%
)\ can be $\square $ (or $C$) independent of $\alpha $?

(2) Which terms of the sequence $(G_{n})_{n\geq 0}$ (or $(H_{n})_{n\geq 0}$)
can be $\square $ (or $C$) with admissible choice of $\alpha $?

(3) Which terms of the sequence $(G_{n})_{n\geq 0}$ (or $(H_{n})_{n\geq 0}$)
can not be $\square $ (or $C$) independent of $\alpha $?

Here again we only consider the case $N=8$, for the convenience of the
reader, we have given the other cases in Appendix B.

\begin{theorem}
Let $E_{8}$ be a Tate normal form of an elliptic curve with a torsion point $%
P=(0,0)$ of order $8$. Let $(G_{n})_{n\geq 0}$ and $(H_{n})_{n\geq 0}$ be
the sequences generated by the numerators of the $x$- and $y$-coordinates of
the multiples of $P$ as in (\ref{a}). Let $G_{n}$ and $H_{n}$ $\neq 0$.%
\newline
1.\newline
\hspace*{0.35cm}$(i)$ $\bullet $ If $n\equiv 0~(8)$, then $G_{n}=\square $
for all $\alpha $ $\neq 0$, $1$, \newline
\hspace*{0.25cm} \ $\ \ \ \bullet $ if $n\equiv 2$, $6~(8)$, then $%
G_{n}=\square $ iff $(\alpha -1)(2\alpha -1)$ $=\square $, \\[0.02cm]
\hspace*{0.35cm} $\ \ \ \bullet $ otherwise $G_{n}\neq \square $ for all $%
\alpha $ $\neq 0$, $1$. \newline
\hspace*{0.25cm}$(ii)$ $\bullet $ If $n\equiv 0~(24)$, then $G_{n}=C$ for
all $\alpha $ $\neq 0$, $1$, \\[0.02cm]
\hspace*{0.35cm} $\ \ \ \bullet $ if $n\equiv 2$, $10$, $14$, $22~(24)$,
then $G_{n}=C$ iff $\alpha =C$,\quad\ \newline
\hspace*{0.25cm} \ $\ \ \ \bullet $ if $n\equiv 8$, $16~(24)$, then $G_{n}=C$
iff $\alpha -1=C$, \\[0.02cm]
\hspace*{0.35cm} $\ \ \ \bullet $ otherwise $G_{n}\neq C$ for all $\alpha $ $%
\neq 0$, $1$. \newline
2.\newline
\hspace*{0.25cm}$(i)$ $\bullet $ If $n\equiv 0$, $4$, $8$, $12~(16)$, then $%
H_{n}=\square $ for all $\alpha $ $\neq 0$, $1$, \newline
\hspace*{0.25cm} \ $\ \ \ \bullet $ if $n\equiv 3~(16)$, then $H_{n}=\square 
$ iff $\alpha -1$ $=\square $, \\[0.02cm]
\hspace*{0.35cm} $\ \ \ \bullet $ if $n\equiv 5$, $7~(16)$, then $%
H_{n}=\square $ iff $2\alpha -1$ $=\square $, \newline
\hspace*{0.25cm} \ $\ \ \ \bullet $ if $n\equiv 11~(16)$, then $%
H_{n}=\square $ iff $\alpha $ $=\square $, \\[0.02cm]
\hspace*{0.35cm} $\ \ \ \bullet $ otherwise $H_{n}\neq \square $ for all $%
\alpha $ $\neq 0$, $1$. \newline
\hspace*{0.25cm}$(ii)$ $\bullet $ If $n\equiv 0$ $(8)$, then $H_{n}=C$ for
all $\alpha $ $\neq 0$, $1$, \\[0.02cm]
\hspace*{0.35cm} $\ \ \ \bullet $ otherwise $H_{n}\neq C$ for all $\alpha $ $%
\neq 0$, $1$.
\end{theorem}

\begin{proof}
We give the proof only for the sequence $(G_{n})_{n\geq 0}$ as the proof for 
$(H_{n})_{n\geq 0}$ is similar.

1. \textit{i}. We note that $G_{n}=0$ for\ $n\equiv 1$, $7~(\func{mod}8)$,
by (\ref{13}). It can easily be seen that $G_{n}=\square $ for every $\alpha 
$ $\neq 0$, $1$ for $n\equiv 0~(8)$, by using (\ref{13}).

If $n\equiv 2$, $6~(8)$, then $G_{n}=\square $ iff 
\begin{equation*}
(\alpha -1)(2\alpha -1)=\square 
\end{equation*}%
by (\ref{13}). This equation leads to 
\begin{equation}
(4\alpha -3)^{2}-8\beta ^{2}=1  \label{01}
\end{equation}%
or 
\begin{equation}
(4\alpha -3)^{2}+8\beta ^{2}=1\text{,}  \label{02}
\end{equation}%
where $\beta $ is a non zero-integer. The last equation is trivial equation
and the solutions of this equation do not provide any acceptable $\alpha $.
The first equation leads to Pell equation%
\begin{equation*}
\tau ^{2}-8\beta ^{2}=1
\end{equation*}%
where $\tau =4\alpha -3$. The solutions of this equation are $(3,1)$, $(17,6)
$, $...$ . Note that only the solutions of the form $\tau +3\equiv 0$ $(4)$
give the acceptable $\alpha $, and their number is infinite.

If $n\equiv 3$, $5~(8)$, then $G_{n}=\square $ iff 
\begin{equation*}
\alpha (\alpha -1)(2\alpha -1)=\square \text{,}
\end{equation*}%
and if $n\equiv 4~(8)$, then $G_{n}=\square $ iff 
\begin{equation*}
\alpha (\alpha -1)=\square
\end{equation*}%
by (\ref{13}). These last two equations lead to trivial equations 
\begin{equation}
(2\alpha -1)^{2}\pm \beta ^{2}=1  \label{025}
\end{equation}%
where $\beta $ is a non-zero integer. The solutions of these trivial
equations do not provide any acceptable $\alpha $, which completes the proof
of (\textit{i}).

\textit{ii}. If $1$, $7$, $9$, $15$, $17$, $23~(24)$, then $G_{n}=0$, if $%
n\equiv 0~(24)$, then $G_{n}=C$ for every $\alpha $ $\neq 0$, $1$, if $n$ $%
\equiv 2$, $10$, $14$, $22~(24)$, then $G_{n}=C$ iff $\alpha =C$, and if $%
n\equiv 8$, $16~(24)$, then $G_{n}=C$ iff $\alpha -1=C$, by (\ref{13}).

If $n\equiv 3$, $21~(24)$, then $G_{n}=C$ iff 
\begin{equation*}
\alpha ^{2}(\alpha -1)(2\alpha -1)=C\text{,}
\end{equation*}%
and if $n$ $\equiv 4$, $20~(24)$, then $G_{n}=C$ iff%
\begin{equation*}
\alpha (\alpha -1)=C\text{,}
\end{equation*}%
if $n$ $\equiv 6$, $18~(24)$, then $G_{n}=C$ iff%
\begin{equation*}
\alpha ^{2}(\alpha -1)=C\text{,}
\end{equation*}%
and if $n$ $\equiv $ $12$ $(24)$, then $G_{n}=C$ iff 
\begin{equation*}
\alpha (\alpha -1)^{2}=C
\end{equation*}%
by (\ref{13}). These equations lead to 
\begin{equation*}
\alpha (\alpha -1)=C\text{.}
\end{equation*}%
This equation leads to trivial equation%
\begin{equation}
\beta _{1}^{3}-\beta _{2}^{3}=1\text{,}  \label{026}
\end{equation}%
where $\alpha =\beta _{1}^{3}$, $\alpha -1=\beta _{2}^{3}$, and $\beta _{1}$%
, $\beta _{2}$ are non-zero integers. The solutions of this equation do not
provide any acceptable $\alpha $.

If $n$ $\equiv 5$, $11$, $19$, $13~(24)$, then $G_{n}=C$ iff 
\begin{equation*}
\alpha ^{2}(2\alpha -1)=C
\end{equation*}%
by (\ref{13}), or equivalently 
\begin{equation*}
\alpha (2\alpha -1)=C\text{.}
\end{equation*}%
The last equation leads to classical equation\footnote{{\small The equation }%
$x^{3}+2y^{3}=1${\small \ has the integer solution }$(x,y)=(-1,1)${\small ,
hence, by Theorem }$5${\small , Chapter 24 of \cite{MOR} can not have
further solutions with }$xy\neq 0${\small .}} 
\begin{equation}
2\beta _{1}^{3}+(-\beta _{2})^{3}=1\text{,}  \label{26}
\end{equation}%
where $\alpha =\beta _{1}^{3}$, $2\alpha -1=\beta _{2}^{3}$, and $\beta _{1}$%
, $\beta _{2}$ are non-zero integers. The solution of this equation does not
provide any acceptable $\alpha $, which completes the proof of (\textit{ii}).
\end{proof}

In the following theorem we determine square and cube terms in the sequence $%
(G_{n})_{n\geq 0}$ associated to elliptic curves $E_{2}$ and $E_{3}$,
respectively, and square and cube terms in the sequence $(H_{n})_{n\geq 0}$
associated to elliptic curve $E_{3}$. The proof is similar to the proof of
Theorem 5.1.

\begin{theorem}
$E_{N}$ be an elliptic curve with the torsion point $P=(0,0)$ of order $N$
as in (\ref{30}) and (\ref{31}). Let $(G_{n})_{n\geq 0}$ and $(H_{n})_{n\geq
0}$ be the sequences generated by the numerators of the $x$- and $y$%
-coordinates of the multiples of $P$ as in (\ref{a}), and let $G_{n}$ $\neq
0 $.\newline
{}1. Let $N=2$. \newline
\hspace*{0.35cm}$(i)$ $\bullet $ $G_{n}=\square $ for every non-zero $a_{4}$%
. \newline
\hspace*{0.20cm}$(ii)$ $\bullet $ If $n\equiv 0~(6)$, then $G_{n}=C$ for
every non-zero $a_{4}$, \\[0.02cm]
\hspace*{0.35cm} $\ \ \ \bullet $ otherwise $G_{n}=C$ iff $a_{4}$ $=C$.\quad 
\newline
2. Let $N=3$. \newline
\hspace*{0.35cm}$(i)$ $\bullet $ $G_{n}=\square $ for every non-zero $a_{3}$%
. \newline
\hspace*{0.20cm}$(ii)$ $\bullet $ $G_{n}=C$\textit{\ }for every non-zero $%
a_{3}$.{}\newline
{} $(iii)$ $\bullet $ If $n\equiv 0$, $2~(6)$, then $H_{n}=\square $ for
every non-zero $a_{3}$, \newline
\hspace*{0.15cm} \ $\ \ \ \bullet $ otherwise $H_{n}=\square $ iff $a_{3}$ $%
=\square $, \newline
$\ (iv)$ $\bullet $ If $n\equiv 0~(3)$, then $H_{n}=C$ for every non-zero $%
a_{3}$, \\[0.02cm]
\hspace*{0.20cm} $\ \ \ \bullet $ otherwise $H_{n}=C$ iff $a_{3}$ $=C$.
\end{theorem}

\bigskip \appendix{}

\section{\protect\appendix{}}

In the following theorems we determine general terms of the sequences $%
(G_{n})_{n\geq 0}$ and $(H_{n})_{n\geq 0}$ associated to an elliptic curve
in Tate normal form with a torsion point $P=(0,0)$ of order $N$. The proofs
are similar to the proof of Theorem 4.3.

\begin{theorem}
Let $E_{N}$ be a Tate normal form of an elliptic curve with a torsion point $%
P=(0,0)$ of order $N$. Let $(G_{n})_{n\geq 0}$ be the sequence generated by
the numerators of the $x$-coordinates of the multiples of $P$ as in (\ref{a}%
). Let $\zeta ,\lambda ,\theta $ be as in (\ref{be}) and (\ref{la}). Then
the general term of the sequence $(G_{n})_{n\geq 0}$ can be given by the
following formulas:\newline
1. If $N=4$, then 
\begin{equation}
{\small G}_{n}{\small \ =}\left\{ 
\begin{array}{cc}
0\text{,} & \text{if }n\text{ is odd} \\ 
\alpha ^{\{3n^{2}/4\}}\text{,} & \text{if }n\text{ is even, }%
\end{array}%
\right.  \label{5}
\end{equation}%
\newline
where $\alpha \neq 0$. \newline
2. If $N=5$, then%
\begin{equation}
{\small G}_{n}{\small \ =}\left\{ 
\begin{array}{cc}
0\text{,} & \text{if }{\small n}\text{ }{\small \equiv 1,4~(5)} \\ 
\alpha ^{\{(4n^{2}-a)/5\}}\text{,} & \text{otherwise,}%
\end{array}%
\right.  \label{7}
\end{equation}%
where $\alpha \neq 0$, and 
\begin{equation*}
a~{\small =}\left\{ 
\begin{array}{ll}
0\text{,} & \text{if }n\equiv 0~(5) \\ 
1\text{,} & \text{if }n\equiv 2,3~(5)\text{.}%
\end{array}%
\right.
\end{equation*}%
\newline
3. If $N=6$, then%
\begin{equation}
{\small G}_{n}~{\small =}\left\{ 
\begin{array}{cc}
0\text{,} & \text{if }{\small n}\text{ }{\small \equiv 1,5~(6)} \\ 
{\small \alpha }^{\{(5n^{2}-a)/6\}}{\small (\alpha +1)}^{\{(2n^{2}+b)/3\}}%
\text{,} & \text{otherwise,}%
\end{array}%
\right.  \label{9}
\end{equation}%
where $\alpha \neq -1$, $0$, and 
\begin{equation*}
~a~{\small =}\left\{ 
\begin{array}{ll}
0\text{,} & \text{if }n\equiv 0~(6) \\ 
2\text{,} & \text{if }n\equiv 2,4~(6) \\ 
3\text{,} & \text{if }n\equiv 3~(6)\text{,}%
\end{array}%
\text{ }\right. b~{\small =}\left\{ 
\begin{array}{ll}
0\text{,} & \text{if }n\equiv 0,3~(6) \\ 
1\text{,} & \text{if }n\equiv 2,4~(6)\text{.}%
\end{array}%
\right. \newline
\end{equation*}%
\newline
4. If $N=7$, then%
\begin{equation}
{\small G}_{n}~{\small =}\left\{ 
\begin{array}{cc}
0\text{,} & \text{if }n\equiv 1,6~(7) \\ 
\alpha ^{\{(10n^{2}+a)/7\}}(\alpha -1)^{\{(6n^{2}-b)/7\}}\text{,} & \text{%
otherwise,}%
\end{array}%
\right.  \label{11}
\end{equation}%
where $\alpha \neq 0$, $1$, and 
\begin{equation*}
a=\left\{ 
\begin{array}{cc}
0\text{,} & \text{if }n\equiv 0~(7) \\ 
2\text{,} & \text{if }n\equiv 2,5~(7) \\ 
1\text{,} & \text{if }n\equiv 3,4~(7)\text{,}%
\end{array}%
\right. b=\left\{ 
\begin{array}{cc}
0\text{,} & \text{if }n\equiv 0~(7) \\ 
3\text{,} & \text{if }n\equiv 2,5~(7) \\ 
5\text{,} & \text{if }n\equiv 3,4~(7)\text{.}%
\end{array}%
\right.
\end{equation*}%
5. If $N=9$, then 
\begin{equation}
{\small G}_{n}~{\small =}\left\{ 
\begin{array}{cc}
{\small 0}\text{,} & \text{if }{\small n}\text{ }{\small \equiv 1,8~(9)} \\ 
{\small \alpha }^{\{(14n^{2}-a)/9\}}{\small (\alpha -1)}^{\{(8n^{2}-b)/9\}}%
\eta ^{\{(2n^{2}+c)/3\}}\text{,} & \text{otherwise}%
\end{array}%
\right.  \label{15}
\end{equation}%
where $\alpha \neq 0$, $1$, $\eta =\alpha ^{2}-\alpha +1$, and 
\begin{equation*}
a=\left\{ 
\begin{array}{ll}
0\text{,} & \text{if }n\equiv 0,3,6~(9) \\ 
2\text{,} & \text{if }n\equiv 2,7~(9) \\ 
-1\text{,} & \text{if }n\equiv 4,5~(9)\text{,}%
\end{array}%
\right. b=\left\{ 
\begin{array}{ll}
0\text{,} & \text{if }n\equiv 0~(9) \\ 
5\text{,} & \text{if }n\equiv 2,7~(9) \\ 
9\text{,} & \text{if }n\equiv 3,6~(9) \\ 
11\text{,} & \text{if }n\equiv 4,5~(9)\text{,}%
\end{array}%
\right.
\end{equation*}%
\begin{equation*}
c=\left\{ 
\begin{array}{ll}
0\text{,} & \text{if }n\equiv 0,3,6~(9) \\ 
1\text{,} & \text{if }n\equiv 2,4,5,7~(9)\text{.}%
\end{array}%
\right.
\end{equation*}%
6. If $N=10$, then%
\begin{equation}
{\small G}_{n}~{\small =}\left\{ 
\begin{array}{cc}
{\small 0}\text{,} & \text{if }{\small n}\text{ }{\small \equiv 1,9~(10)} \\ 
\begin{array}{c}
{\small \alpha }^{\{(21n^{2}+a)/10\}}{\small (\alpha -1)}^{\{(9n^{2}-b)/10\}}
\\ 
\times {\small (2\alpha -1)}^{\{(4n^{2}-c)/5\}}{\small \zeta }%
^{\{(5n^{2}+d)/5\}}\text{,}%
\end{array}
& \text{otherwise,}%
\end{array}%
\right.  \label{17}
\end{equation}%
\newline
where $\alpha \neq 0$, $1$,%
\begin{equation*}
a=\left\{ 
\begin{array}{ll}
0\text{,} & \text{if }n\equiv 0~(10) \\ 
6\text{,} & \text{if }n\equiv 2,8~(10) \\ 
1\text{,} & \text{if }n\equiv 3,7~(10) \\ 
4\text{,} & \text{if }n\equiv 4,6~(10) \\ 
5\text{,} & \text{if }n\equiv 5~(10)\text{,}%
\end{array}%
\right. b=\left\{ 
\begin{array}{ll}
0\text{,} & \text{if }n\equiv 0~(10)\text{,} \\ 
6\text{,} & \text{if }n\equiv 2,8~(10) \\ 
11\text{,} & \text{if }n\equiv 3,7~(10) \\ 
14\text{,} & \text{if }n\equiv 4,6~(10) \\ 
15\text{,} & \text{if }n\equiv 5~(10)\text{,}%
\end{array}%
\right.
\end{equation*}%
and 
\begin{equation*}
c=\left\{ 
\begin{array}{ll}
0\text{,} & \text{if }n\equiv 0,5~(10) \\ 
1\text{,} & \text{if }n\equiv 2,3,7,8~(10) \\ 
-1\text{,} & \text{if }n\equiv 4,6~(10)\text{,}%
\end{array}%
\right. d=\left\{ 
\begin{array}{ll}
0\text{,} & \text{if }n\equiv 0,2,4,6,8~(10) \\ 
1\text{,} & \text{if }n\equiv 3,5,7~(10)\text{.}%
\end{array}%
\right.
\end{equation*}%
\newline
7. If $N=12$, then%
\begin{equation}
{\small G}_{n}~{\small =}\left\{ 
\begin{array}{cc}
{\small 0}\text{,} & \text{if }{\small n}\text{ }{\small \equiv 1,11~(12)}
\\ 
\begin{array}{c}
\varepsilon {\small \alpha }^{\{(n^{2}-~{\small a})/6\}}{\small (\alpha -1)}%
^{\{(59n^{2}+{\small b})/12\}} \\ 
\times {\small (2\alpha -1)}^{\{(n^{2}-{\small c})/12\}}{\small \lambda }%
^{\{(3n^{2}+{\small d})/4\}}{\small \theta }^{\{(2n^{2}+{\small e})/3\}}%
\text{,}%
\end{array}
& \text{otherwise,}%
\end{array}%
\right.  \label{19}
\end{equation}%
where $\alpha \neq 0$, $1$,%
\begin{equation*}
\varepsilon =\left\{ 
\begin{array}{cc}
+1\text{,} & \text{if }n\equiv 0,2,3,9,10~(12) \\ 
-1\text{,} & \text{if }n\equiv 4,5,6,7,8~(12)\text{,}%
\end{array}%
\right.
\end{equation*}%
\begin{equation*}
a=\left\{ 
\begin{array}{cc}
0\text{,} & \text{if }n\equiv 0~(12) \\ 
4\text{,} & \text{if }n\equiv 2,10~(12) \\ 
9\text{,} & \text{if }n\equiv 3,9~(12) \\ 
10\text{,} & \text{if }n\equiv 4,8~(12) \\ 
13\text{,} & \text{if }n\equiv 5,7~(12) \\ 
12\text{,} & \text{if }n\equiv 6~(12)\text{,}%
\end{array}%
\right. \text{ }b=\left\{ 
\begin{array}{cc}
0\text{,} & \text{if }n\equiv 0~(12) \\ 
4\text{,} & \text{if }n\equiv 2,10~(12) \\ 
9\text{,} & \text{if }n\equiv 3,9~(12) \\ 
16\text{,} & \text{if }n\equiv 4,8~(12) \\ 
1\text{,} & \text{if }n\equiv 5,7~(12) \\ 
24\text{,} & \text{if }n\equiv 6\text{ }(12)\text{,}%
\end{array}%
\right.
\end{equation*}%
\begin{equation*}
c=\left\{ 
\begin{array}{cc}
0\text{,} & \text{if }n\equiv 0,6~(12)\text{,} \\ 
4\text{,} & \text{if }n\equiv 2,4,8,10~(12) \\ 
9\text{,} & \text{if }n\equiv 3,9~(12) \\ 
1\text{,} & \text{if }n\equiv 5,7~(12)\text{,}%
\end{array}%
\right. d=\left\{ 
\begin{array}{cc}
0\text{,} & \text{if }n\equiv 0,2,4,6,8,10~(12) \\ 
1\text{,} & \text{if }n\equiv 3,5,7,9~(12)\text{,}%
\end{array}%
\right.
\end{equation*}%
and%
\begin{equation*}
e=\left\{ 
\begin{array}{cc}
0\text{,} & \text{if }n\equiv 0,3,6,9~(12) \\ 
1\text{,} & \text{if }n\equiv 2,4,5,7,8,10~(12)\text{.}%
\end{array}%
\right.
\end{equation*}
\end{theorem}

\begin{theorem}
Let $E_{N}$ be a Tate normal form of an elliptic curve with a torsion point $%
P=(0,0)$ of order $N$. Let $(H_{n})_{n\geq 0}$ be the sequence generated by
the numerators of the $y$-coordinates of the multiples of $P$ as in (\ref{a}%
). Let $\zeta ,\lambda ,\theta $ be as in (\ref{be}) and (\ref{la}). Then
the general term of the sequence $(H_{n})_{n\geq 0}$ can be given by the
following formulas:\newline
1. If $N=4$, then%
\begin{equation}
{\small H}_{n}{\small \ =}\left\{ 
\begin{array}{cc}
0\text{,} & \text{if }n\equiv 1,2~(4) \\ 
\varepsilon \alpha ^{\{(9n^{2}-a)/8\}}\text{,} & \text{otherwise,}%
\end{array}%
\right.
\end{equation}%
where $\alpha \neq 0$, and 
\begin{equation*}
{\small \varepsilon =}\left\{ 
\begin{array}{ll}
+1\text{,} & \text{if }n\equiv 0~(8) \\ 
-1\text{,} & \text{if }n\equiv 3,4,7~(8)\text{,}\ 
\end{array}%
\right. a=\left\{ 
\begin{array}{ll}
0\text{,} & \text{if }n\equiv 0~(4)\  \\ 
1\text{,} & \text{if }n\equiv 3~(4)\text{.}%
\end{array}%
\right.
\end{equation*}%
\newline
2. If $N=5$, then%
\begin{equation}
{\small H}_{n}=\left\{ 
\begin{array}{cc}
0\text{,} & \text{if }n\equiv 1,3~(5)\text{,} \\ 
\varepsilon \alpha ^{\{(6n^{2}-a)/5\}}\text{,} & \text{otherwise,}%
\end{array}%
\right.
\end{equation}%
where $\alpha \neq 0$, and 
\begin{equation*}
{\small \varepsilon =}\left\{ 
\begin{array}{ll}
+1\text{,} & \text{if }n\equiv 0,4,7~(10) \\ 
-1\text{,} & \text{if }n\equiv 2,5,9~(10)\text{,}%
\end{array}%
\right. a=\left\{ 
\begin{array}{cc}
0\text{,} & \text{if }n\equiv 0~(5) \\ 
-1\text{,} & \text{if }n\equiv 2~(5) \\ 
1\text{,} & \text{if }n\equiv 4~(5)\text{.}%
\end{array}%
\right.
\end{equation*}%
\newline
3. If $N=6$, then%
\begin{equation}
{\small H}_{n}~{\small =}\left\{ 
\begin{array}{cc}
0\text{,} & \text{if }{\small n}\equiv {\small 1,4~(6)}\text{,} \\ 
\varepsilon \alpha ^{\{(5n^{2}-a)/4\}}(\alpha +1)^{n^{2}}\text{,} & \text{%
otherwise,}%
\end{array}%
\right.
\end{equation}%
where $\alpha \neq -1$, $0$, and%
\begin{equation*}
{\small \varepsilon =}\left\{ 
\begin{array}{ll}
+1\text{,} & \text{if }n\equiv 0,5,6,9~(12)\text{ \ \ }\  \\ 
-1\text{,} & \text{if }n\equiv 2,3,8,11~(12)\text{,}%
\end{array}%
\right. a=\left\{ 
\begin{array}{ll}
0\text{,} & \text{if }n\equiv 0,2~(6) \\ 
1\text{,} & \text{if }n\equiv 3,5~(6)\text{.}%
\end{array}%
\right.
\end{equation*}%
\newline
4. If $N=7$, then%
\begin{equation}
{\small H}_{n}{\small ~=}\left\{ 
\begin{array}{cc}
{\small 0}\text{,} & \text{if }{\small n}\text{ }{\small \equiv 1,5~(7)} \\ 
{\small \varepsilon \alpha }^{\{(15n^{2}-a)/7\}}{\small (\alpha -1)}%
^{\{(9n^{2}-b)/7\}}\text{,} & \text{otherwise,}%
\end{array}%
\right.
\end{equation}%
where $\alpha \neq 0$, $1$,%
\begin{equation*}
{\small \varepsilon =}\left\{ 
\begin{array}{ll}
+1\text{,} & \text{if }n\equiv 0,4,7,11~(14)\text{ \ }\  \\ 
-1\text{,} & \text{if }n\equiv 2,3,6,9,10,13~(14)\text{,}%
\end{array}%
\right.
\end{equation*}%
and%
\begin{equation*}
a=\left\{ 
\begin{array}{cc}
0\text{,} & \text{if }n\equiv 0~(7) \\ 
-3\text{,} & \text{if }n\equiv 2~(7) \\ 
2\text{,} & \text{if }n\equiv 3~(7) \\ 
-5\text{,} & \text{if }n\equiv 4~(7) \\ 
1\text{,} & \text{if }n\equiv 6~(7)\text{,}%
\end{array}%
\right. b=\left\{ 
\begin{array}{cc}
0\text{,} & \text{if }n\equiv 0~(7) \\ 
1\text{,} & \text{if }n\equiv 2~(7) \\ 
4\text{,} & \text{if }n\equiv 3,4~(7) \\ 
2\text{,} & \text{if }n\equiv 6~(7)\text{.}%
\end{array}%
\right.
\end{equation*}%
\newline
5. If $N=9$, then%
\begin{equation}
{\small H}_{n}{\small \ =}\left\{ 
\begin{array}{cc}
{\small 0}\text{,} & \text{if }{\small n}\text{ }{\small \equiv 1,7~(9)} \\ 
{\small \varepsilon \alpha }^{\{(7n^{2}+a)/3\}}{\small (\alpha -1)}%
^{\{(4n^{2}-b)/3\}}\eta ^{(n^{2}+c)}\text{,} & \text{otherwise}%
\end{array}%
\right.
\end{equation}%
where $\alpha \neq 0$, $1$, $\eta =\alpha ^{2}-\alpha +1$, 
\begin{equation*}
\varepsilon =\left\{ 
\begin{array}{ll}
+1\text{,} & \text{if }n\equiv 0,4,5,8,11,12,15~(18)\text{ \ \ }\  \\ 
-1\text{,} & \text{if }n\equiv 2,3,6,9,13,14,17~(18)\text{,}%
\end{array}%
\right.
\end{equation*}%
and%
\begin{equation*}
a=\left\{ 
\begin{array}{ll}
{\small 0}\text{,} & \text{if }{\small n}\text{ }{\small \equiv 0,3~(9)} \\ 
{\small 2}\text{,} & \text{if }{\small n}\text{ }{\small \equiv 2,5~(9)} \\ 
-{\small 1}\text{,} & \text{if }{\small n}\text{ }{\small \equiv 4,8~(9)} \\ 
{\small 3}\text{,} & \text{if }{\small n}\text{ }{\small \equiv 6~(9)}\text{%
{\small ,}}%
\end{array}%
\right. b=\left\{ 
\begin{array}{ll}
{\small 0}\text{,} & \text{if }{\small n}\text{ }{\small \equiv 0~(9)} \\ 
{\small 1}\text{,} & \text{if }{\small n}\text{ }{\small \equiv 2,8~(9)} \\ 
{\small 3}\text{,} & \text{if }{\small n}\text{ }{\small \equiv 3,6~(9)} \\ 
{\small 4}\text{,} & \text{if }{\small n}\text{ }{\small \equiv 4,5~(9)}%
\text{,}%
\end{array}%
\right. c=\left\{ 
\begin{array}{ll}
{\small 1}\text{,} & \text{if }{\small n}\text{ }{\small \equiv 4~(9)} \\ 
{\small 0}\text{,} & \text{{\small otherwise.}}%
\end{array}%
\right.
\end{equation*}%
\newline
6. If $N=10$, then\newline
\begin{equation}
{\small H}_{n}~{\small =}\left\{ 
\begin{array}{cc}
{\small 0}\text{,} & \text{if }{\small n}\text{ }{\small \equiv 1,8~(10)} \\ 
\begin{array}{c}
{\small \varepsilon \alpha }^{\{(63n^{2}+a)/20\}}{\small (\alpha -1)}%
^{\{(27n^{2}-b)/20\}} \\ 
\times {\small (2\alpha -1)}^{\{(6n^{2}+c)/5\}}{\small \zeta }%
^{\{(15n^{2}+d)/4\}}\text{,}%
\end{array}
& \text{otherwise,}%
\end{array}%
\right.
\end{equation}%
where $\alpha \neq 0$, $1$,%
\begin{equation*}
{\small \varepsilon =}\left\{ 
\begin{array}{ll}
+1\text{,} & \text{if }n\text{ }\equiv 0,4,5,9,10,13,14,17~(20)\text{ \ \ }\ 
\\ 
-1\text{,} & \text{if }n\text{ }\equiv 2,3,6,7,12,15,16,19~(20)\text{,}%
\end{array}%
\right.
\end{equation*}%
and%
\begin{eqnarray*}
a &=&\left\{ 
\begin{array}{cc}
0\text{,} & \text{if }n\equiv 0~(10) \\ 
8\text{,} & \text{if }n\equiv 2~(10) \\ 
-7\text{,} & \text{if }n\equiv 3~(10) \\ 
32\text{,} & \text{if }n\equiv 4~(10) \\ 
25\text{,} & \text{if }n\equiv 5~(10) \\ 
-8\text{,} & \text{if }n\equiv 6~(10) \\ 
13\text{,} & \text{if }n\equiv 7~(10) \\ 
-3\text{,} & \text{if }n\equiv 9~(10)\text{,}%
\end{array}%
\right. \text{ \ \ \ \ }b=\left\{ 
\begin{array}{cc}
0\text{,} & \text{if }n\equiv 0~(10)\text{,} \\ 
8\text{,} & \text{if }n\equiv 2~(10) \\ 
23\text{,} & \text{if }n\equiv 3\text{, }7~(10) \\ 
32\text{,} & \text{if }n\equiv 4\text{, }6~(10) \\ 
35\text{,} & \text{if }n\equiv 5~(10) \\ 
7\text{,} & \text{if }n\equiv 9~(10)\text{,}%
\end{array}%
\right. \\
c &=&\left\{ 
\begin{array}{cc}
0\text{,} & \text{if }n\equiv 0,5~(10) \\ 
1\text{,} & \text{if }n\equiv 2,3,7~(10) \\ 
-1\text{,} & \text{if }n\equiv 4,9~(10) \\ 
4\text{,} & \text{if }n\equiv 6~(10)\text{,}%
\end{array}%
\right. d=\left\{ 
\begin{array}{ll}
0\text{,} & \text{if }n\equiv 0,2,4,6~(10) \\ 
1\text{,} & \text{otherwise.}%
\end{array}%
\right.
\end{eqnarray*}%
\newline
7. If $N=12$, then%
\begin{equation}
{\small H}_{n}~{\small =}\left\{ 
\begin{array}{cc}
{\small 0}\text{,} & \text{if }{\small n}\text{ }{\small \equiv 1,10~(12)}
\\ 
\begin{array}{c}
{\small \varepsilon \alpha }^{\{(n^{2}-a)/4\}}{\small (\alpha -1)}%
^{\{(59n^{2}+b)/8\}} \\ 
\times {\small (2\alpha -1)}^{\{(n^{2}-c)/8\}}{\small \lambda }^{\{(9n^{2}-%
{\small d})/8\}}{\small \theta }^{(n^{2}+e)}\text{,}%
\end{array}
& \text{otherwise,}%
\end{array}%
\right.  \label{20}
\end{equation}%
where $\alpha \neq 0$, $1$,%
\begin{equation*}
\varepsilon =\left\{ 
\begin{array}{ll}
+1\text{,} & \text{if }n\equiv 0,3,5,7,8,11,14,16,18,21~(24)\text{ \ \ }\ 
\\ 
-1\text{,} & \text{if }n\equiv 2,4,6,9,12,15,17,19,20,23~(24)\text{,}%
\end{array}%
\right.
\end{equation*}%
and%
\begin{eqnarray*}
a &=&\left\{ 
\begin{array}{cc}
0\text{,} & \text{if }n\equiv 0~(12) \\ 
4\text{,} & \text{if }n\equiv 2~(12) \\ 
5\text{,} & \text{if }n\equiv 3~(12) \\ 
8\text{,} & \text{if }n\equiv 4,8~(12) \\ 
13\text{,} & \text{if }n\equiv 5~(12) \\ 
12\text{,} & \text{if }n\equiv 6~(12) \\ 
9\text{,} & \text{if }n\equiv 7,9~(12) \\ 
1\text{,} & \text{if }n\equiv 11~(12)\text{,}%
\end{array}%
\right. \text{ \ \ \ }b=\left\{ 
\begin{array}{cc}
0\text{,} & \text{if }n\equiv 0,4~(12) \\ 
-4\text{,} & \text{if }n\equiv 2~(12) \\ 
13\text{,} & \text{if }n\equiv 3~(12) \\ 
5\text{,} & \text{if }n\equiv 5,11~(12) \\ 
12\text{,} & \text{if }n\equiv 6~(12) \\ 
-3\text{,} & \text{if }n\equiv 7,9~(12) \\ 
16\text{,} & \text{if }n\equiv 8~(12)\text{,}%
\end{array}%
\right. \\
c &=&\left\{ 
\begin{array}{cc}
0\text{,} & \text{if }n\equiv 0,4,8~(12) \\ 
4\text{,} & \text{if }n\equiv 2,6~(12) \\ 
1\text{,} & \text{if }n\equiv 3,11~(12) \\ 
9\text{,} & \text{if }n\equiv 5,9~(12) \\ 
-7\text{,} & \text{if }n\equiv 7~(12)\text{,}%
\end{array}%
\right. d=\left\{ 
\begin{array}{cc}
0\text{,} & \text{if }n\equiv 0,4,8~(12) \\ 
-4\text{,} & \text{if }n\equiv 2,6~(12) \\ 
1\text{,} & \text{if }n\equiv 3,7,11~(12) \\ 
-7\text{,} & \text{if }n\equiv 5,9~(12)\text{,}%
\end{array}%
\right.
\end{eqnarray*}%
and 
\begin{equation*}
e=\left\{ 
\begin{array}{ll}
1\text{,} & \text{if }n\equiv 4,7~(12) \\ 
0\text{,} & \text{otherwise.}%
\end{array}%
\right. \newline
\end{equation*}
\end{theorem}

\section{\protect\appendix{}}

In the following theorems we determine square and cube terms in the
sequences $(G_{n})_{n\geq 0}$ and $(H_{n})_{n\geq 0}$ associated to an
elliptic curve in Tate normal form with a torsion point $P=(0,0)$ of order $%
N $. The proofs are similar to the proof of Theorem 5.1.

\begin{theorem}
Let $E_{N}$ be a Tate normal form of an elliptic curve with a torsion point $%
P=(0,0)$ of order $N$. Let $(G_{n})_{n\geq 0}$ be the sequence generated by
the numerators of the $x$-coordinates of the multiples of $P$ as in (\ref{a}%
), and let $G_{n}$ $\neq 0$.\newline
{}1. Let $N=4$. \newline
\hspace*{0.35cm}$(i)$ $\bullet $ If $n\equiv 0~(4)$, then $G_{n}=\square $
for all non-zero $\alpha $, \newline
\hspace*{0.25cm} \ $\ \ \ \bullet $ otherwise $G_{n}=\square $ iff $\alpha $ 
$=\square $. \newline
\hspace*{0.25cm}$(ii)$ $\bullet $ $G_{n}=C$ for all non-zero $\alpha $.
\quad \newline
2. Let $N=5$. \newline
\hspace*{0.35cm}$(i)$ $\bullet $ If $n\equiv 0~(5)$, then $G_{n}=\square $
for all non-zero $\alpha $, \newline
\hspace*{0.25cm} \ $\ \ \ \bullet $ otherwise $G_{n}=\square $ iff $\alpha $ 
$=\square $. \newline
\hspace*{0.25cm}$(ii)$ $\bullet $ If $n\equiv 0$, $2$, $7$, $8$, $13~(15)$,
then $G_{n}=C$ for all non-zero $\alpha $, \\[0.02cm]
\hspace*{0.35cm} $\ \ \ \bullet $ otherwise $G_{n}=C$ iff $\alpha $ $=C$%
.\quad \newline
3. Let $N=6$. \newline
\hspace*{0.35cm}$(i)$ $\bullet $ If $n\equiv 0~(6)$, then $G_{n}=\square $
for all $\alpha $ $\neq -1$, $0$, \newline
\hspace*{0.25cm} \ $\ \ \ \bullet $ if $n\equiv 3~(6)$, then $G_{n}=\square $
iff $\alpha $ $=\square $, \\[0.02cm]
\hspace*{0.35cm} $\ \ \ \bullet $ otherwise $G_{n}\neq \square $ for all $%
\alpha $ $\neq -1$, $0$. \newline
\hspace*{0.25cm}$(ii)$ $\bullet $ If $n\equiv 0$, $2$, $6$, $12$, $16~(18)$,
then $G_{n}=C$ for all $\alpha $ $\neq -1$, $0$, \quad\ \newline
\hspace*{0.25cm} \ $\ \ \ \bullet $ if $n\equiv 3$, $9$, $15~(18)$, then $%
G_{n}=C$ iff $\alpha =C$, \\[0.02cm]
\hspace*{0.35cm} $\ \ \ \bullet $ otherwise $G_{n}\neq C$ for all $\alpha $ $%
\neq -1$, $0$. \newline
4. Let $N=7$. \newline
\hspace*{0.35cm}$(i)$ $\bullet $ If $n\equiv 0~(7)$, then $G_{n}=\square $
for all $\alpha $ $\neq 0$, $1$, \newline
\hspace*{0.25cm} \ $\ \ \ \bullet $ if $n\equiv 2$, $5~(7)$, then $%
G_{n}=\square $ iff $\alpha -1$ $=\square $, \\[0.02cm]
\hspace*{0.35cm} $\ \ \ \bullet $ otherwise $G_{n}\neq \square $ for all $%
\alpha $ $\neq 0$, $1$. \newline
\hspace*{0.25cm}$(ii)$ $\bullet $ If $n\equiv 0$, $2$, $5$, $16$, $19~(21)$,
then $G_{n}=C$ for all $\alpha $ $\neq 0$, $1$, \\[0.02cm]
\hspace*{0.35cm} $\ \ \ \bullet $ if $n\equiv 7$, $9$, $12$, $14~(21)$, then 
$G_{n}=C$ iff $\alpha =C$, \\[0.02cm]
\hspace*{0.35cm} $\ \ \ \bullet $ otherwise $G_{n}\neq C$ for all $\alpha $ $%
\neq 0$, $1$. \newline
5. Let $N=9$. \newline
\hspace*{0.35cm}$(i)$ $\bullet $ If $n\equiv 0~(9)$, then $G_{n}=\square $
for all $\alpha $ $\neq 0$, $1$, \newline
\hspace*{0.25cm} \ $\ \ \ \bullet $ if $n\equiv 3$, $6~(9)$, then $%
G_{n}=\square $ iff $\alpha -1$ $=\square $, \\[0.02cm]
\hspace*{0.35cm} $\ \ \ \bullet $ otherwise $G_{n}\neq \square $ for all $%
\alpha $ $\neq 0$, $1$. \newline
\hspace*{0.25cm}$(ii)$ $\bullet $ If $n\equiv 0$, $2$, $9$, $18$, $25~(27)$,
then $G_{n}=C$ for all $\alpha $ $\neq 0$, $1$, \\[0.02cm]
\hspace*{0.35cm} $\ \ \ \bullet $ if $n\equiv 5$, $22~(27)$, then $G_{n}=C$
iff $\alpha ^{2}-\alpha +1=C$,\quad\ \\[0.02cm]
\hspace*{0.35cm} $\ \ \ \bullet $ otherwise $G_{n}\neq C$ for all $\alpha $ $%
\neq 0$, $1$. \newline
6. Let $N=10$. \newline
\hspace*{0.35cm}$(i)$ $\bullet $ \textit{If }$n\equiv 0~(10)$\textit{, then }%
$G_{n}=\square $\textit{\ for all }$\alpha $\textit{\ }$\neq 0$\textit{, }$1$%
, \newline
\hspace*{0.25cm} \ $\ \ \ \bullet $ \textit{if }$n\equiv 4$, $6~(10)$\textit{%
, then }$G_{n}=\square $\textit{\ iff }$(\alpha -1)(2\alpha -1)$\textit{\ }$%
=\square $\textit{,} \\[0.02cm]
\hspace*{0.35cm} $\ \ \ \bullet $ otherwise $G_{n}\neq \square $ for all $%
\alpha $ $\neq 0$, $1$. \newline
\hspace*{0.25cm}$(ii)$ $\bullet $ \textit{If }$n\equiv 0~(30)$\textit{, then 
}$G_{n}=C$\textit{\ for all }$\alpha $\textit{\ }$\neq 0$\textit{, }$1$, \\[%
0.02cm]
\hspace*{0.35cm} $\ \ \ \bullet $ \textit{if }$n\equiv 2$, $8$, $22$, $%
28~(30)$\textit{, then }$G_{n}=C$\textit{\ iff }$\alpha ^{2}-3\alpha +1=C$%
,\quad\ \newline
\hspace*{0.25cm} \ $\ \ \ \bullet $ \textit{if }$n\equiv 12$, $18~(30)$%
\textit{, then }$G_{n}=C$\textit{\ iff }$2\alpha -1=C$, \\[0.02cm]
\hspace*{0.35cm} $\ \ \ \bullet $ otherwise $G_{n}\neq C$ for all $\alpha $ $%
\neq 0$, $1$. \newline
7. Let $N=12$. \newline
\hspace*{0.35cm}$(i)$ $\bullet $ If $n\equiv 0~(12)$, then $G_{n}=\square $
for all $\alpha $ $\neq 0$, $1$, \\[0.02cm]
\hspace*{0.35cm} $\ \ \ \bullet $ otherwise $G_{n}\neq \square $ for all $%
\alpha $ $\neq 0$, $1$. \newline
\hspace*{0.25cm}$(ii)$ $\bullet $ If $n\equiv 0$, $12$, $24$ $(36)$, then $%
G_{n}=C$ for all $\alpha $ $\neq 0$, $1$, \\[0.02cm]
\hspace*{0.35cm} $\ \ \ \bullet $ if $n\equiv 16$, $20~(36)$, then $G_{n}=C$
iff $\alpha =C$,\quad\ \\[0.02cm]
\hspace*{0.35cm} $\ \ \ \bullet $ if $n\equiv 2$, $34~(36)$, then $G_{n}=C$
iff $\alpha -1=C$, \quad\ \\[0.02cm]
\hspace*{0.35cm} $\ \ \ \bullet $ otherwise $G_{n}\neq C$ for all $\alpha $ $%
\neq 0$, $1$.
\end{theorem}

\begin{theorem}
Let $E_{N}$ be a Tate normal form of an elliptic curve with a torsion point $%
P=(0,0)$ of order $N$. Let $(H_{n})_{n\geq 0}$ be the sequence generated by
the numerators of the $y$-coordinates of the multiples of $P$ as in (\ref{a}%
), and let $H_{n}$ $\neq 0$.\newline
{}1. Let $N=4$. \newline
\hspace*{0.35cm}$(i)$ $\bullet $ If $n\equiv 0$, $3$, $4~(8)$, then $%
H_{n}=\square $ for all non-zero $\alpha $, \newline
\hspace*{0.25cm} \ $\ \ \ \bullet $ otherwise $H_{n}=\square $ iff $\alpha $ 
$=\square $. \newline
\hspace*{0.25cm}$(ii)$ $\bullet $ If $n\equiv 0~(4)$, then $H_{n}=C$ for all
non-zero $\alpha $, \\[0.02cm]
\hspace*{0.35cm} $\ \ \ \bullet $ otherwise $H_{n}=C$ iff $\alpha $ $=C$%
.\quad \newline
2. Let $N=5$. \newline
\hspace*{0.35cm}$(i)$ $\bullet $ If $n\equiv 0~(5)$, then $H_{n}=\square $
for all non-zero $\alpha $, \newline
\hspace*{0.25cm} \ $\ \ \ \bullet $ otherwise $H_{n}=\square $ iff $\alpha $ 
$=\square $. \newline
\hspace*{0.25cm}$(ii)$ $\bullet $ If $n\equiv 0~(5)$, then $H_{n}=C$ for all
non-zero $\alpha $, \\[0.02cm]
\hspace*{0.35cm} $\ \ \ \bullet $ otherwise $H_{n}=C$ iff $\alpha $ $=C$. 
\newline
3. Let $N=6$. \newline
\hspace*{0.35cm}$(i)$ $\bullet $ If $n\equiv 0$, $8~(12)$, then $%
H_{n}=\square $ for all $\alpha $ $\neq -1$, $0$, \newline
\hspace*{0.25cm} \ $\ \ \ \bullet $ if $n\equiv 2$, $6~(12)$, then $%
H_{n}=\square $ iff $\alpha $ $=\square $, \\[0.02cm]
\hspace*{0.35cm} $\ \ \ \bullet $ otherwise $H_{n}\neq \square $ for all $%
\alpha $ $\neq -1$, $0$. \newline
\hspace*{0.25cm}$(ii)$ $\bullet $ If $n\equiv 0~(6)$, then $H_{n}=C$ for all 
$\alpha $ $\neq -1$, $0$, \quad\ \newline
\hspace*{0.25cm} \ $\ \ \ \bullet $ if $n\equiv 3~(6)$, then $H_{n}=C$ iff $%
\alpha =C$, \\[0.02cm]
\hspace*{0.35cm} $\ \ \ \bullet $ otherwise $H_{n}\neq C$ for all $\alpha $ $%
\neq -1$, $0$. \newline
4. Let $N=7$. \newline
\hspace*{0.35cm}$(i)$ $\bullet $ If $n\equiv 0$, $9$, $10~(14)$, then $%
H_{n}=\square $ for all $\alpha $ $\neq 0$, $1$, \newline
\hspace*{0.25cm} \ $\ \ \ \bullet $ if $n\equiv 4$, $6~(14)$, then $%
H_{n}=\square $ iff $\alpha $ $=\square $, \\[0.02cm]
\hspace*{0.25cm} \ $\ \ \ \bullet $ if $n\equiv 11$, $13~(14)$, then $%
H_{n}=\square $ iff $\alpha -1$ $=\square $, \\[0.02cm]
\hspace*{0.35cm} $\ \ \ \bullet $ otherwise $H_{n}\neq \square $ for all $%
\alpha $ $\neq 0$, $1$. \newline
\hspace*{0.25cm}$(ii)$ $\bullet $ If $n\equiv 0$ $(7)$, then $H_{n}=C$ for
all $\alpha $ $\neq 0$, $1$, \\[0.02cm]
\hspace*{0.35cm} $\ \ \ \bullet $ if $n\equiv 2~(7)$, then $H_{n}=C$ iff $%
\alpha -1=C$, \\[0.02cm]
\hspace*{0.35cm} $\ \ \ \bullet $ otherwise $H_{n}\neq C$ for all $\alpha $ $%
\neq 0$, $1$. \newline
5. Let $N=9$. \newline
\hspace*{0.35cm}$(i)$ $\bullet $ If $n\equiv 0$, $13$, $14~(18)$, then $%
H_{n}=\square $ for all $\alpha $ $\neq 0$, $1$, \newline
\hspace*{0.25cm} \ $\ \ \ \bullet $ if $n\equiv 2$, $12~(18)$, then $%
H_{n}=\square $ iff $\alpha -1$ $=\square $, \\[0.02cm]
\hspace*{0.35cm} $\ \ \ \bullet $ otherwise $H_{n}\neq \square $ for all $%
\alpha $ $\neq 0$, $1$. \newline
\hspace*{0.25cm}$(ii)$ $\bullet $ If $n\equiv 0$ $(9)$, then $H_{n}=C$ for
all $\alpha $ $\neq 0$, $1$, \\[0.02cm]
\hspace*{0.35cm} $\ \ \ \bullet $ if $n\equiv 3$ $(9)$, then $H_{n}=C$ iff $%
\alpha -1=C$,\quad\ \\[0.02cm]
\hspace*{0.35cm} $\ \ \ \bullet $ otherwise $H_{n}\neq C$ for all $\alpha $ $%
\neq 0$, $1$. \newline
6. Let $N=10$. \newline
\hspace*{0.25cm} $(i)$ $\bullet $ If $n\equiv 0$, $5$, $15$, $16~(20)$, then 
$H_{n}=\square $ for all $\alpha $ $\neq 0$, $1$, \newline
\hspace*{0.25cm} \ $\ \ \ \bullet $ if $n\equiv 4$, $12~(20)$, then $%
H_{n}=\square $ iff $2\alpha -1$ $=\square $, \\[0.02cm]
\hspace*{0.35cm} $\ \ \ \bullet $ if $n\equiv 3$, $13~(20)$, then $%
H_{n}=\square $ iff $(\alpha -1)(2\alpha -1)$ $=\square $, \\[0.02cm]
\hspace*{0.35cm} $\ \ \ \bullet $ otherwise $H_{n}\neq \square $ for all $%
\alpha $ $\neq 0$, $1$. \\[0.02cm]
\hspace*{0.25cm}$(ii)$ $\bullet $ \textit{If }$n\equiv 0~(10)$\textit{, then 
}$H_{n}=C$\textit{\ for all }$\alpha $\textit{\ }$\neq 0$\textit{, }$1$,%
\hspace*{0.35cm} $\ $\\[0.02cm]
\hspace*{0.35cm} $\ \ \ \bullet $ otherwise $H_{n}\neq C$ for all $\alpha $ $%
\neq 0$, $1$. \newline
7. Let $N=12$. \newline
\hspace*{0.35cm}$(i)$ $\bullet $ If $n\equiv 0$, $8$, $12$, $19$, $20~(24)$,
then $H_{n}=\square $ for all $\alpha $ $\neq 0$, $1$, \newline
\hspace*{0.25cm} \ $\ \ \ \bullet $ if $n\equiv 4$, $16~(24)$, then $%
H_{n}=\square $ iff $2\alpha -2\alpha ^{2}-1$ $=\square $, \\[0.02cm]
\hspace*{0.35cm} $\ \ \ \bullet $ otherwise $H_{n}\neq \square $ for all $%
\alpha $ $\neq 0$, $1$. \newline
\hspace*{0.25cm}$(ii)$ $\bullet $ \textit{If }$n\equiv 0~(12)$\textit{, then 
}$H_{n}=C$\textit{\ for all }$\alpha $\textit{\ }$\neq 0$\textit{, }$1$,%
\hspace*{0.35cm} $\ $\\[0.02cm]
\hspace*{0.35cm} $\ \ \ \bullet $ otherwise $H_{n}\neq C$ for all $\alpha $ $%
\neq 0$, $1$.
\end{theorem}

\end{document}